\documentclass[12pt,reqno]{amsart}
\usepackage{etex}
\usepackage{dsfont,amsmath,amsfonts,amscd,amssymb,graphicx,mathrsfs,eufrak,upgreek}
\usepackage[dvips,all,arc,curve,color,frame]{xy}
\usepackage[usenames]{color}
\usepackage{setspace}
\usepackage{array,multirow,booktabs,longtable}

\makeatletter
\newcommand{\customlabel}[3]{#3\def\@currentlabel{#2}\label{#1}}
\makeatother

\usepackage[final]{optional}

\usepackage[colorlinks,hyperfootnotes=false]{hyperref} 
\newcommand{\hrefsf}[2]{\href{#1}{\textsf{#2}}}
\newcommand{\hrefsfopt}[3]{{#1}\hrefsf{#2}{#3}} 

\usepackage{versions}
\includeversion{arXiv}

\newcommand{\marginparstretch}{0.6}
\let\oldmarginpar\marginpar
\renewcommand\marginpar[1]{\-\oldmarginpar[\framebox{\setstretch{\marginparstretch}\begin{minipage}{\marginparwidth}{\raggedleft\tiny #1}\end{minipage}}]{\framebox{\setstretch{\marginparstretch}\begin{minipage}{\marginparwidth}{\raggedright\tiny #1}\end{minipage}}}}

\usepackage[colorlinks]{hyperref}
\usepackage{tikz,tikz-cd,mathrsfs}
\usetikzlibrary{arrows,positioning,calc}

\usepackage{pgfplots}
\pgfplotsset{compat=1.8}
\usetikzlibrary{decorations.markings}

\tikzset{
  mid arrow/.style={postaction={decorate,decoration={
        markings,
        mark=at position .5 with {\arrow{stealth}}
      }}},
}

\tikzset{
        Point/.style={circle,draw=black,circle,fill=black,inner sep=0pt, minimum size=2pt},
        DynkinBlack/.style={circle,draw=black,circle,fill=black,inner sep=0pt, minimum size=4pt},
         DynkinWhite/.style={circle,draw=black,circle,fill=white,inner sep=0pt, minimum size=4pt},
        vertex/.style={circle,fill=black,inner sep=1pt,outer sep=8pt},
        star/.style={circle,fill=yellow,inner sep=0.75pt,outer sep=0.75pt},
        gap/.style={inner sep=0.5pt,fill=white}}

\tikzstyle{mybox} = [draw=black, fill=blue!10, very thick,
    rectangle, rounded corners, inner sep=10pt, inner ysep=20pt]
\tikzstyle{boxtitle} =[fill=blue!50, text=white,rectangle,rounded corners]

\newcommand{\arrow}[2][20]
 {
  \hspace{-5pt}
  \begin{tikzpicture}
   \node (A) at (0,0) {};
   \node (B) at (#1pt,0) {};
   \draw [#2] (A) -- (B);
  \end{tikzpicture}
  \hspace{-5pt}
 }

\newcommand{\arrowsplit}{0.4ex}

\newtheorem{theorem}{Theorem}[section]
\newtheorem{keytheorem}{Theorem}[section]

\newtheorem{proposition}[theorem]{Proposition}
\newtheorem{lemma}[theorem]{Lemma}
\newtheorem{definition}[theorem]{Definition}
\newtheorem{corollary}[theorem]{Corollary}

\theoremstyle{definition} 

\newtheorem{example}[theorem]{Example}
\newtheorem{remark}[theorem]{Remark}
\newtheorem*{keyremark}{Remark}

\newtheorem{setup}[theorem]{Setup}

\theoremstyle{remark}
\newtheorem*{acknowledgements}{Acknowledgements}

\newcommand{\defined}[1]{\emph{#1}}

\def\Hom{\mathop{\rm Hom}\nolimits}

\def\D{\mathop{\rm{D}^{}}\nolimits}

\def\K{\mathop{\rm{K}^{}}\nolimits}
\newcommand{\Kbig}[1]{\K\!\big({#1}\big)}

\def\Id{\mathop{\sf{Id}}\nolimits}
\def\Idmatrix{\mathds{1}}
\def\gitQuot{/\!\!/}
\newcommand\gitQuotStab[1]{\gitQuot_{\!#1}}
\newcommand\gitQuotStabBrief[1]{\gitQuot{#1}}

\newcommand{\cC}{\mathcal{C}}
\newcommand{\cD}{\mathcal{D}}
\newcommand{\cE}{\mathcal{E}}

\newcommand{\cG}{\mathcal{G}}

\newcommand{\cM}{\mathcal{M}}
\newcommand{\cN}{\mathcal{N}}
\newcommand{\cO}{\mathcal{O}}
\newcommand{\cP}{\mathcal{P}}

\newcommand{\cW}{\mathcal{W}}

\newcommand\compose{\circ}
\newcommand\placeholder{-}
\newcommand\dual{\vee}

\newlength\tempWidth

\newcommand\KS{Kapranov--Schechtman}
\newcommand\KandS{Kapranov and Schechtman}
\newcommand\linzn{\cM}
\newcommand\sphFun{\mathsf{S}}
\newcommand\twistFun{\mathsf{T}}
\newcommand\cotwistFun{\mathsf{C}}
\newcommand\field{\mathds{k}}

\newcommand\reals{\mathbb{R}}
\newcommand\labelpos{1.3}
\newcommand\labelposrelax{0.2}
\newcommand\axisposrelax{0.4}
\newcommand\sphaxiscut{0.4}
\newcommand\arrowbend{55}
\newcommand\labeladjust{0.05}
\newcommand{\axes}[3]{
 \draw[->, gray!50] (-2.3-#2+#3,0) to (2.4+#2,0) node[right, #1] {$\reals$};
 \draw[->, gray!50] (0,-1.7) to (0,1.8) node[above, #1] {$i\reals$};
}

\newcommand{\monodromyPic}[6]{
\begin{tikzpicture}[node distance=1cm, auto, line width=0.5pt]

\axes{white}{#6}{0}

 \node (mid) at (0,-\labeladjust) {#1};
 \node (pos) at (-\labelpos-#5,-\labeladjust)  {#2};
 \node (neg) at (+\labelpos+#5,-\labeladjust) {#3};

 \draw[->, bend left=\arrowbend] (pos.north) to (mid.north west);
 \draw[->, bend left=\arrowbend] (mid.north east) to (neg.north);
 \draw[->, dashed, bend left=\arrowbend] (pos.north) to node[above] {\scriptsize #4} (neg.north);

 \draw[->, bend left=\arrowbend] (neg.south) to (mid.south east);
 \draw[->, bend left=\arrowbend] (mid.south west) to (pos.south);
 \draw[->, dashed, bend left=\arrowbend] (neg.south) to (pos.south);

\end{tikzpicture}
}

\newcommand{\monodromyOnlyPic}[4]{
\begin{tikzpicture}[node distance=1cm, auto, line width=0.5pt]

\axes{white}{\axisposrelax}{0}

 \node (mid) at (0,-\labeladjust) {#1};
 \node (pos) at (-\labelpos-\labelposrelax,-\labeladjust)  {#2};
 \node (neg) at (+\labelpos+\labelposrelax,-\labeladjust) {#3};

 \draw[->, bend left=\arrowbend] (pos.north) to node[above] {\scriptsize #4} (neg.north);

 \draw[->, bend left=\arrowbend] (neg.south) to (pos.south);

\end{tikzpicture}
}

\newcommand{\objectsOnlyPic}[4]{
\begin{tikzpicture}[node distance=1cm, auto, line width=0.5pt]

\axes{white}{0}{0}

 \node (mid) at (0,0) {#1};
 \node (pos) at (-\labelpos,0)  {#2};
 \node (neg) at (+\labelpos,0) {#3};

\end{tikzpicture}
}

\newcommand{\sphericalPic}[4]{
\begin{tikzpicture}[node distance=1cm, auto, line width=0.5pt]

\axes{white}{\axisposrelax}{2*\sphaxiscut}

 \node (mid) at (0,-2*\labeladjust) {#1};
 \node (neg) at (+\labelpos+\labelposrelax,-\labeladjust) {#3};

 \draw[->, bend left=\arrowbend] (mid.north) to node[above] {\scriptsize #4} (neg.north);

 \draw[->, bend left=\arrowbend] (neg.south) to (mid.south);

\end{tikzpicture}
}

\newcommand{\cutPic}[4]{
\begin{tikzpicture}[node distance=1cm, auto, line width=0.5pt]

\axes{black}{0}{0}

\draw[line width=0.5pt] (0,0) circle (1.5);

\draw[fill=black] (0,0) circle (.2ex);
\draw[fill=black] (-1.5,0) circle (.2ex);
\draw[fill=black] (+1.5,0) circle (.2ex);

\draw (0,0) to [bend right=20] (1.5,0);
\draw (0,0) to [bend right=20] (-1.5,0);

 \node (label) at (0.8,-0.4) {#4};

 \node (pos) at (-1.8,0.2)  {\scriptsize $-1$};
 \node (neg) at (+1.8,0.2) {\scriptsize $+1$};

\end{tikzpicture}
}

\newcommand{\halfCutPic}[4]{
\begin{tikzpicture}[node distance=1cm, auto, line width=0.5pt]

\axes{black}{0}{\sphaxiscut}

\draw[line width=0.5pt] (0,0) circle (1.5);

\draw[fill=black] (0,0) circle (.2ex);
\draw[fill=black] (+1.5,0) circle (.2ex);

\draw (0,0) to [bend right=20] (1.5,0);

 \node (label) at (0.8,-0.4) {#4};

 \node (neg) at (+1.8,0.2) {\scriptsize $+1$};

\end{tikzpicture}
}

\begin{document}

\title[Perverse schobers and wall crossing]{\phantom{text}\vspace{-0.8cm}Perverse schobers and wall crossing}
\author{W.\ Donovan}
\address{Kavli IPMU (WPI), UTIAS, University of Tokyo, Kashiwa, Chiba, Japan}
\email{will.donovan@ipmu.jp}
\begin{abstract}
For a balanced wall crossing in geometric invariant theory, there exist derived equivalences between the corresponding GIT quotients if certain numerical conditions are satisfied. Given such a wall crossing, I~construct a perverse sheaf of categories on a disk, singular at a point, with half-monodromies recovering these equivalences, and with behaviour at the singular point controlled by a GIT quotient stack associated to the~wall. Taking complexified Grothendieck groups gives a perverse sheaf of vector spaces: I characterise when this is an intersection cohomology complex of a local system on the punctured disk.
\end{abstract}

\subjclass[2010]{Primary 14F05; 
Secondary 14E05, 
14L24, 
18E30, 
32S60. 
}
\thanks{The author is supported by World Premier International Research Center Initiative (WPI Initiative), MEXT, Japan, and JSPS KAKENHI Grant Number~JP16K17561.}
\maketitle
\parindent 20pt
\parskip 0pt


\tableofcontents

\section{Introduction}
\label{section intro}

\KandS{} have initiated a program to define and study perverse sheaves of triangulated categories, known as perverse schobers~\cite{KS2}. It is expected that these new objects will be crucial to the emerging field of categorified birational geometry. In this paper, I give a large class of examples in the form of `spherical pairs'. These arise as categorifications of perverse sheaves of vector spaces on a disk, possibly singular at a point. I~will explain, following \KS{}, different such categorifications, depending on a choice of `skeleton' consisting of a~union of disjoint arcs from the disk boundary to the point: for~two~arcs, the categorification is a spherical pair; for~one~arc, it is a \mbox{spherical functor}.

\newpage
My main result, Theorem~\ref{keytheorem}, constructs spherical pairs for many birational maps coming from GIT wall crossings. My simplest examples come from an orbifold flop of local~$\mathbb{P}^1$, and an Atiyah flop of a resolved conifold. These are given in an expository Section~\ref{section simple}: further examples, namely flops of local $\mathbb{P}^{\mathrm{odd}}$, and standard flops, are given in the final Section~\ref{section examples}.

\begin{keyremark} Harder, Katzarkov, and Liu study a different notion of perverse sheaves of categories, in relation to rationality: see~\cite{HKL} and references therein.\end{keyremark}

\subsection{Discussion} I outline some of my results informally, before giving details in Section~\ref{section results}.

\subsubsection*{Spherical pairs} Take embeddings of a pair of categories of interest $\cE_\pm$ into a single category $\cE_0$. This data may be viewed as a perverse sheaf of categories on a disk as follows. Take $K$ to be a skeleton with two arcs joining~$\pm 1$~to~$0$, as illustrated in Figure~\ref{figure sph pair}. For a perverse sheaf of vector spaces, possibly singular at $0$, the local cohomology with support in $K$ is concentrated in some fixed degree by \mbox{`purity'}: the categories $\cE_\pm$ and $\cE_0$ should then be seen as categorifications of the stalks of this sheaf of local cohomology at~$\pm 1$~and~$0$ respectively, and the embeddings as categorifications of maps between them. A spherical pair consists of such categories $\cE_\pm$ and $\cE_0$ along with embeddings satisfying natural conditions: these are given in Section~\ref{subsection.sph_pair}.

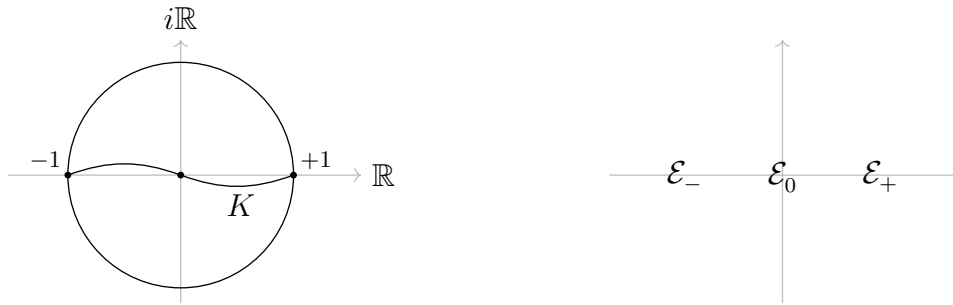
\begin{figure}[h]
\begin{center}
\begin{tikzpicture}
 \node at (0,0) {$\cutPic{$0$}{$-1$}{$+1$}{$K$}$};
 \node at (8,0) {$\objectsOnlyPic{$\cE_0^{\vphantom{+}}$}{$\cE_-^{\vphantom{+}}$}{$\cE_+^{\vphantom{+}}$}{}$};
 \end{tikzpicture}
\end{center}
\caption{Spherical pair.}\label{figure sph pair}
\end{figure}

\subsubsection*{Setup} To obtain an example of a spherical pair, consider a GIT wall crossing given by the data of a variety $X$ with an action of a group $G$ and a family of linearizations $\cM_t$ for small $t\in\mathbb{R}$, such that the GIT quotients $X \gitQuotStab{t}\, G$ are constant for $t<0$ and $t>0$ respectively, but there exist strict semistables for~$t=0$. Assume furthermore that this wall crossing is `simple balanced': this implies in particular that the change in the GIT quotient due to crossing the wall is controlled by a single one-parameter subgroup~$\lambda$ of~$G$. In this case, assuming a numerical condition given in Theorem~\ref{keytheorem}(\ref{keytheorem assumption 2}), there exist derived equivalences between the quotients $X \gitQuotStabBrief{\pm}$ by results of Halpern-Leistner~\cite{HL} and Ballard--Favero--Katzarkov~\cite{BFK}.

\subsubsection*{Construction} Given this setup, I construct a spherical pair or, equivalently, a perverse sheaf of categories on the disk where $\cE_\pm$ are the bounded derived categories $\D(X \gitQuotStabBrief{\pm}),$ and $\cE_0$ is a subcategory $\cC$ of the bounded derived category $\D(X^{\mathrm{ss}}(\cM_0)/G)$ of the stack of semistables associated to the wall. This data is indicated in Figure~\ref{figure sph pair eg}.

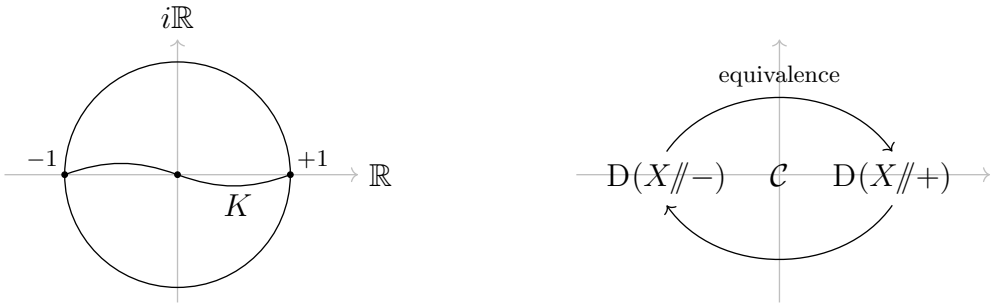
\begin{figure}[h]
\begin{center}
\begin{tikzpicture}
 \node at (0,0) {$\cutPic{$0$}{$-1$}{$+1$}{$K$}$};
 \node at (8,0) {$\monodromyOnlyPic{$\mathcal{C}\vphantom{\gitQuotStabBrief{+}}$}{$\D(X \gitQuotStabBrief{-})$}{$\D(X \gitQuotStabBrief{+})$}{equivalence}$};
\end{tikzpicture}
\end{center}
\caption{Spherical pair example.}\label{figure sph pair eg}
\end{figure}

\begin{keyremark} It is natural to view the disk on which our perverse sheaf of categories is defined as lying in the complexification of the GIT parameter space~$\mathbb{R}$: intriguingly, this suggests the existence of natural perverse schobers on (appropriate compactifications of) Bridgeland stability spaces.
\end{keyremark}

\begin{keyremark} For comparison, note that Bodzenta and Bondal construct spherical pairs for certain varieties $Y_\pm$ related by flops of families of curves, some of which arise from GIT wall crossings as above, where $\cE_\pm = \D(Y_\pm)$, and $\cE_0$~is obtained from the fibre product $Y_- \times_B Y_+$ over the base $B$ of the flop~\cite{BB}. Current work of Bondal, Kapranov, and Schechtman will generalize this construction to interesting webs of flops, in particular flops of Springer resolutions, using a notion of perverse schobers on $\mathbb{C}^n$ singular along a real hyperplane arrangement~\cite{BKS}.\end{keyremark}

\subsubsection*{Monodromy} Spherical pairs induce equivalences $\cE_- \!\leftrightarrow \cE_+$, categorifications of the half-monodromy functions for a perverse sheaf of vector spaces. (In~the example above, these half-monodromy functors recover the known derived equivalences between the quotients $X \gitQuotStabBrief{\pm}$.) By composing these, the spherical pair also induce symmetries of the categories $\cE_\pm$, categorifications of monodromy actions. Different choices of cut give different descriptions of this monodromy. For instance, it turns out that for a skeleton $K'$ with only one~arc, as in Figure~\ref{fig dim 1 cats}, the categorification is a spherical functor, and that spherical pairs determine spherical functors which recover the symmetry of~$\cE_+$ above as a spherical twist.

\subsubsection*{Spherical functor} The spherical functor determined by our spherical pair is shown in Figure~\ref{fig dim 1 cats}. The category $\cD$ is given in Theorem~\ref{keytheorem}(\ref{keytheorem 3}) in terms of sheaves on the unstable locus of one of the GIT quotients; the functor itself is described in Corollary~\ref{corollary sph functor}. These spherical functors have appeared previously in work of Halpern-Leistner and Shipman~\cite{HLShipman}: our spherical pair construction provides a new viewpoint on their results.

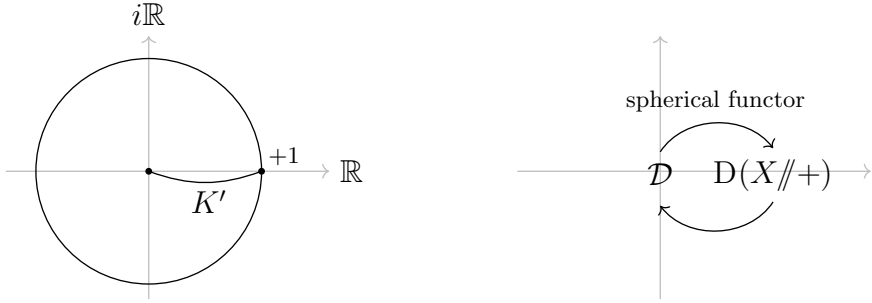
\begin{figure}[h]
\begin{center}
\begin{tikzpicture}
 \node at (0,0) {$\halfCutPic{$0$}{$-1$}{$+1$}{$K'$}$};
 \node at (7,0) {$\sphericalPic{$\mathcal{D}_{\vphantom{\gitQuotStabBrief{+}}}$}{$\D(X \gitQuotStab{-} G)$}{$\D(X \gitQuotStabBrief{+} )$}{spherical functor}$};
\end{tikzpicture}
\end{center}
\caption{Spherical functor.}\label{fig dim 1 cats}

\end{figure}

\begin{keyremark} A technical advantage of working with spherical pairs over spherical functors is that they are easier to construct: we need only show that certain functors are equivalences, without the need to take troublesome functorial cones.\end{keyremark}

\subsection{Results}\label{section results} Consider a GIT wall crossing given by the data of a projective-over-affine variety $X$ with an action of a connected reductive group~$G$ and a pair of linearizations $\cM_\pm$ in the sense of Section~\ref{subsection vgit}. The~unstable loci $X - X^{\mathrm{ss}}(\cM_\pm)$ then come with non-canonical GIT stratifications. Require that this wall crossing is `simple balanced' (Definition~\ref{definition simple balanced}): informally, this means that only a single stratum is affected by crossing the wall. Associated to this stratum is the data of:
\begin{itemize}
\item  a Levi subgroup $L$ of $G$;
\item a one-parameter subgroup $\lambda$ of $L$; and
\item an open subset $Z$ of the $\lambda$-fixed locus in $X$.
\end{itemize}
The following theorem, which is the main result of this paper, constructs a spherical pair for such a wall crossing.

\begin{keytheorem}\label{keytheorem} Take a GIT wall crossing for a variety $X$ with a $G$-action, as~above, which is `simple balanced' as in Definition~\ref{definition simple balanced}. Assume that:

\renewcommand{\theenumi}{\roman{enumi}}
\begin{enumerate}
\item\label{keytheorem assumption 1} the variety $X$ is smooth in a $G$-equivariant neighbourhood of $Z$;
\item\label{keytheorem assumption 2} the canonical sheaf $\omega_X$ has $\lambda$-weight zero on $Z$; and
\item\label{keytheorem assumption 3} the group $G$ is abelian.
\end{enumerate}
Let $\cM_0$ denote a linearization on the wall, and $X^{\mathrm{ss}}(\cM_0)$ the associated semistable locus in $X$. Then the following hold.
\renewcommand{\theenumi}{\arabic{enumi}}
\begin{enumerate}
\item\label{keytheorem 1} \emph{(Theorem \ref{theorem.sph_pair})} For each integer $w$, there exists a subcategory $\cC$ of $\D(X^{\mathrm{ss}}(\cM_0)/G)$ with embeddings
\begin{align*}
\iota_{-\vphantom{+}} & \colon \D(Z/L)^w \longrightarrow \cC \\
\iota_+ & \colon \D(Z/L)^{w+\eta} \longrightarrow \cC
\end{align*}
giving a spherical pair $\cP$~\cite{KS2}, where superscripts denote \mbox{$\lambda$-weight} subcategories, and $\eta$ is the window width of Definition~\ref{definition window width}. 
\item\label{keytheorem 3} \emph{(Corollary \ref{corollary sph functor})} The spherical pair $\cP$ induces a spherical functor
\begin{equation*}
\sphFun \colon \cD  \longrightarrow \D(X \gitQuotStabBrief{+})
\end{equation*}
where $\cD$ denotes the category $\D(Z/L)^w$. 
\item\label{keytheorem 2} \emph{(Corollary \ref{corollary dual pair})} Furthermore, there exist natural embeddings
\[
\epsilon_\pm \colon \D(X \gitQuotStabBrief{\pm}) \longrightarrow \mathcal{C}
\]
which determine a spherical pair in a sense dual to that of \cite{KS2}.
\end{enumerate}
\end{keytheorem}

\begin{keyremark}In many case, the quotients $ X\gitQuotStabBrief{\pm}$ corresponding to a wall crossing are related by a flip. This follows for instance if $G=\mathbb{C}^*$ and the unstable loci have codimension at least two~\cite[Proposition~1.6]{Thaddeus}. Assumption~(\ref{keytheorem assumption 2}) then asserts that this flip is in fact a flop.
\end{keyremark}

\begin{keyremark} It may seem strange to use the notation~$\D(Z/L)$ when $G$~is abelian, as the Levi subgroup $L$ must coincide with $G$ in this case. The notation is taken from Halpern-Leistner and Shipman: we retain it for convenience because, firstly, they give useful alternative descriptions of the categories~$\D(Z/L)^w$ \cite[\mbox{Section~2.1}]{HLShipman} and, secondly, we will have cause to slightly amend one of their lemmas (Lemma~\ref{proposition.iota_adjoints}).
\end{keyremark}

We will see that taking complexified Grothendieck groups of a spherical pair $\cP$ gives a perverse sheaf of vector spaces on a disk. The following theorem shows that this is an intersection cohomology complex when its monodromy satisfies an appropriate condition.

\begin{keytheorem}\label{keytheorem Ktheory} With the setup above, the following hold.
\begin{enumerate}
\item{} \emph{(Proposition \ref{proposition.K_theory})} There exists a perverse sheaf of vector spaces~${}^{\K}\cP$ naturally associated to $\cP$ with generic fibre the complexified Grothendieck group $\K(X \gitQuotStabBrief{+})$; and
\item{} \emph{(Theorem \ref{theorem.Ktheory})} Writing $m$ for the monodromy action on $\K(X \gitQuotStabBrief{+})$, ${}^{\K}\cP$ is an intersection cohomology complex of a local system on the punctured disk if and only if
\begin{equation}\tag{$\ast$}\label{equation equality}
 \operatorname{rk} (m - \Idmatrix) = \operatorname{dim} \Kbig{ \D(Z/L )^w }.
 \end{equation}
\end{enumerate}
\end{keytheorem}
I show in Proposition~\ref{proposition.very_balanced_odd_codim} that a sufficient condition for (\ref{equation equality}) to be satisfied is that the codimension of $Z$ in~$X$ is odd, and both spaces are $G$-equivariantly Calabi--Yau. Nevertheless, (\ref{equation equality}) fails in many simple examples, including an Atiyah flop of a resolved conifold (Section~\ref{section resolved conifold}).

\subsection{Categorified intersection cohomology}

I view the construction of a spherical pair in Theorem~\ref{keytheorem} as an instance of taking `categorified intersection cohomology'. Namely, the equivalences $\cE_- \!\leftrightarrow \cE_+$ naturally define a local system of categories on the punctured disk, and the spherical pair $\cP$ extends this to a perverse sheaf of categories on the whole disk, thus categorifying the way in which intersection cohomology extends local systems of vector spaces to perverse sheaves of vector spaces. Theorem~\ref{keytheorem Ktheory} shows that, under assumption~(\ref{equation equality}), decategorification works as expected. A construction of Segal~\cite{Seg2}, which takes a general autoequivalence $\Psi\colon\cE \rightarrow \cE$ and expresses it as a twist of a spherical functor, may likewise be viewed as an instance of categorified intersection cohomology.

\subsection{Fukaya categories}\label{intro ms}

I briefly mention recent work using perverse schobers to understand Fukaya categories: it would be interesting to link the present paper to some of these studies. \KandS{}'s program is related to Kontsevich's proposal to localise Fukaya categories along singular Lagrangian skeleta, building on work of Seidel~\cite{Seidel}: see~\cite{Kontsevich-Symplectic,Kontsevich-Talk10}, and later~\cite{Kontsevich-Talk15}. This proposal has been extensively pursued by Nadler~\cite{NadlerArb}. Dyckerhoff, Kapranov, Schechtman, and Soibelman~\cite{DKSS} are currently developing a theory of perverse schobers on general Riemann surfaces, following previous work by Dyckerhoff and Kapranov studying topological Fukaya categories of surfaces~\cite{DK}, and by Pascaleff and Sibilla for punctured surfaces~\cite{PS}. Soibelman has discussed perverse schobers in relation to wall-crossing in the stability space associated to the Fukaya category~\cite{Soibelman}. Perverse schobers have appeared furthermore in work of Nadler where they are used, for instance, to understand the Landau-Ginzburg A-model on $\mathbb{C}^n$ with superpotential~$ z_1 \dots z_n $~\cite{Nadler}.

\begin{acknowledgements}I am grateful to M.~Kapranov for many enlightening conversations and valuable suggestions; to V.~Schechtman for helpful explanations of his work; to H.~Iritani for his hospitality and insights in Kyoto; and to T.~Dyckerhoff for his hospitality in Bonn, and his patient explanations of aspects of \cite{DKSS}. I thank E.~Co\c{s}kun, \"{O}.~Gen\c{c}, and \"{O}. Ki\c{s}isel for their warm welcome at Middle East Technical University in Ankara, where I made crucial progress. This project has benefited from numerous discussions over the years with A.~Bondal, E.~Segal, R.~Thomas, and M.~Wemyss. I would also like to thank N.~Addington for helpful comments.\end{acknowledgements}

\newpage
\section{Simple cases}
\label{section simple}
In this section I outline my construction and results in some concrete examples, namely an orbifold flop of local $\mathbb{P}^1$, and an Atiyah flop of a resolved conifold.

\subsection{Perverse sheaves}\label{section diff eqn} These may be thought of as generalizations of the sheaves of solutions to differential equations. As motivation therefore, take a complex number~$c$ and consider the equation \[z \frac{df}{dz} = c f\] where $f$ is a function of a variable $z$ taking values in the  punctured complex disk~$\Delta - 0$. It has a local solution $f(z) = z^c$, however for non-integer~$c$ this solution has monodromy around~$0$, so global solutions are sections of a local system~$L$. This local system may be described by its fibres $E_{\pm}$ at $\pm 1$, along with half-monodromy morphisms between them.

To describe a perverse sheaf $P$ extending $L$ which is defined on all of $\Delta$, though possibly singular at $0$, we may proceed as follows. Take the sheaf of (hyper)cohomology with support $\mathbb{H}^1_K(P)$ on a cut $K$ as in Figure~\ref{figure perverse sheaf}, and write $E_\pm$  and $E_0$ for its stalks at~$\pm 1$~and~$0$ respectively. There are natural maps $E_0 \to E_\pm$, and dual maps in the reverse direction, as shown. These maps satisfy conditions given in Proposition~\ref{proposition.KS_perverse_description}: a spherical pair is then given by a diagrams of categories satisfying analogous conditions given in Definition~\ref{definition sph pair}.

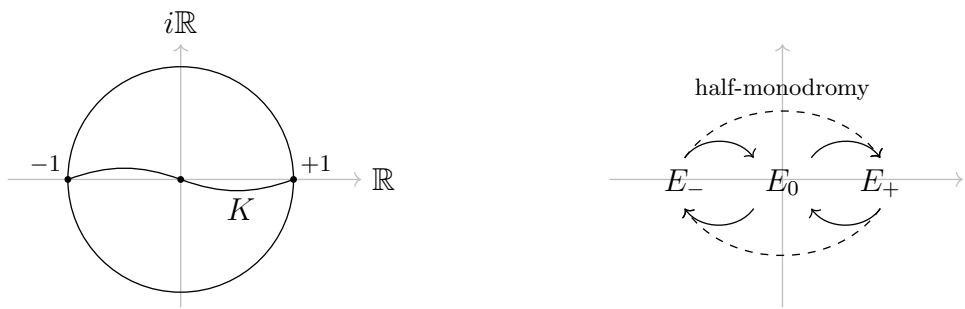
\begin{figure}[h]
\begin{center}
\begin{tikzpicture}
 \node at (0,0) {$\cutPic{$0$}{$-1$}{$+1$}{$K$}$};
 \node at (8,0) {$\monodromyPic{$E_{0\vphantom{+}}$}{$E_{-\vphantom{+}}$}{$E_+$}{half-monodromy}{0}{0}$};
\end{tikzpicture}
\end{center}
\caption{Perverse sheaf on~$\Delta$.}\label{figure perverse sheaf}
\end{figure}

\subsection{Example: local $\mathbb{P}^1$}

Take a $3$-dimensional vector space $X$ with coordinates $(x_1,x_2,y)$, and a $\mathbb{C}^*$-action with weights $(1,1,-2)$, whose orbits are shown in Figure~\ref{figure orbits}. There are two GIT quotients $X \gitQuotStabBrief{\pm}$ of this action, namely:

\begin{enumerate}
\item[\customlabel{quotient 1}{$-$}{($-$)}] the total space $\operatorname{Tot}\!\big(\cO_{\mathbb{P}^1}(-2)\big)$ of a line bundle on $\mathbb{P}^1$; and
\item[\customlabel{quotient 2}{$+$}{($+$)}] an orbifold $\mathbb{C}^2 / C_2$, where the cyclic group $C_2$ acts by $\pm\Idmatrix$.
\end{enumerate}
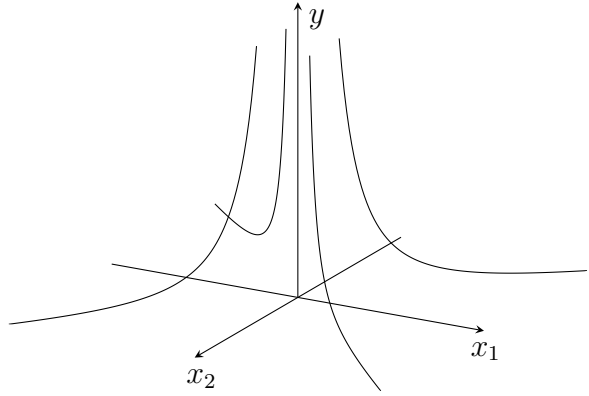
\begin{figure}[h]
\begin{center}
\begin{tikzpicture}[scale=1,line width=0.75pt]
  \begin{axis}[
    view       = {-29}{-25},
    axis lines = middle,
    zmax       = 15,
    height     = 8cm,
    xtick      = \empty,
    ytick      = \empty,
    ztick      = \empty,
    xlabel  = {$x_1$},
    ylabel  = {$\,\,\,x_2$},
    zlabel  = {$y$},
    x label style={at={(axis cs:5,0)},anchor=north},
    y label style={at={(axis cs:0,5)},anchor=north},
  ]
\foreach \s in {0,...,3}
{
  \addplot3+ [
    ytick      = \empty,
    yticklabel = \empty,
    domain     = 1:7,
    samples    = 100,
    samples y  = 0,
    mark       = none,
    black,
  ]
  ( {x*sin(45+90*\s)},{x*cos(45+90*\s)},{13/x^2});
}
  \end{axis}
\end{tikzpicture}
\end{center}
\caption{Orbits in $X$.}\label{figure orbits}
\end{figure}
These quotients arise as open substacks of the quotient stack $X/\mathbb{C}^*$, so we may consider restriction functors as follows.
\[ \operatorname{res}_\pm\colon \D(X/\mathbb{C}^*) \to \D(X \gitQuotStabBrief{\pm} ) \]
It was observed by Segal~\cite{Seg}, following ideas of Kawamata, and Herbst--Hori--Page~\cite{HHP}, later developed into a general theory in~\cite{HL,BFK}, that these are equivalences on certain `window' subcategories
\[ \cW_k = \big\langle\cO(k), \cO(k+1)\big\rangle \]
in $\D(X/\mathbb{C}^*)$, generated by weight line bundles $\cO(k)$ on $X/\mathbb{C}^*$. The key observation for us is that by taking a slightly larger window, for instance
\[ \cC = \big\langle\cO(-1), \cO, \cO(1)\big\rangle, \]
we may then produce a diagram of categories as above, by composing the obvious embeddings of the $\cW_k$ with equivalences, as follows.
\begin{center}
\begin{tikzpicture}[node distance=1cm, auto, line width=0.5pt]
 \node (mid) at (0,0) {$\mathcal{C}$};
 \node (pos) at (-1.5,0)  {$\cW_{-1}$};
 \node at (-2.3,0)  {$\cong$};
 \node at (-3.4,0)  {$\D(X \gitQuotStabBrief{-} )$};
 \node (neg) at (+1.5,0) {$\cW_{0}$};
 \node at (+2.2,0)  {$\cong$};
 \node at (+3.3,0)  {$\D(X \gitQuotStabBrief{+} )$};

 \draw[->, bend left=\arrowbend] (pos.north) to (mid.north west);
 \draw[->, bend left=\arrowbend] (neg.south) to (mid.south east);
\end{tikzpicture}
\end{center}
Functors in the reverse direction are given by adjoints, sketched in Figure~\ref{figure perverse 2dim}.
\vspace{-1cm}

\begin{figure}[h]
\begin{center}
\begin{tikzpicture}
 \node at (0,0) {$\monodromyPic{$\mathcal{C}\vphantom{\gitQuot}$}{$\D(X \gitQuotStabBrief{-} )$}{$\D(X \gitQuotStabBrief{+} )$}{flop functor}{\labelposrelax}{\axisposrelax}$};
 \node at (8,0) {$\monodromyPic{$\mathcal{C}$}{$\big\langle\cO_y(-1)\big\rangle$}{$\big\langle\cO_{\vec{x}}(1)\big\rangle$}{}{\labelposrelax}{\axisposrelax}$};
\end{tikzpicture}
\end{center}
\caption{Spherical pair for local $\mathbb{P}^1$.}\label{figure perverse 2dim}
\end{figure}
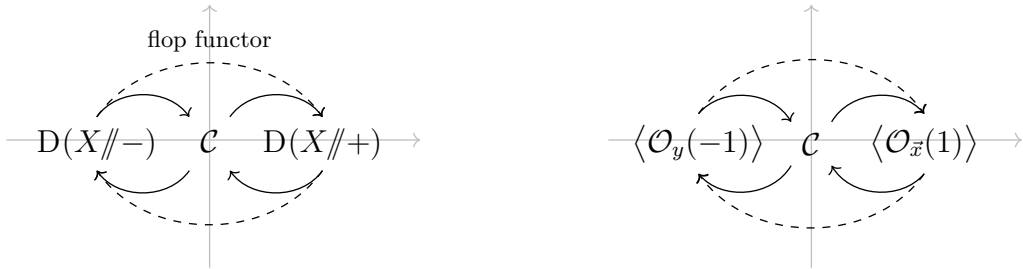

The orthogonal subcategories ${}^\perp \cW_{-1}$ and ${}^\perp \cW_{0}$ in $\cC$ are generated by objects $\cO_y(-1)$ and $\cO_{\vec{x}}(1)$ supported respectively on the $y$-axis, and the $(x_1,x_2)$-plane: they yield the right-hand diagram in Figure~\ref{figure perverse 2dim} which gives a spherical pair $\cP$ in the sense of \KandS{}: the left-hand diagram satisfies dual axioms which are made precise in Corollary~\ref{corollary dual pair}.

Taking complexified Grothendieck groups of the left-hand diagram in Figure~\ref{figure perverse 2dim}, we obtain a perverse sheaf of vector spaces $P={}^{\K}\cP$. This has a nice description as follows. By composing arrows in the diagram we obtain a flop--flop functor~$\mathsf{FF}$ acting on $\D(X \gitQuotStabBrief{+})$, and thence an endomorphism of the complexified Grothendieck group~$\K(X \gitQuotStabBrief{+})$ which determines a local system of vector spaces on the punctured disk~$\Delta-0$: the intersection cohomology complex of this local system recovers~$P$.  This result is explained and generalized in Section~\ref{section intersection cohomology}.

\subsection{Example: Resolved conifold}\label{section resolved conifold}

Now I briefly give a $3$-fold example, where the relation to intersection cohomology is less straightforward. Take a $4$-dimensional vector space $X$ with coordinates $(x_1,x_2,y_1,y_2)$ and a $\mathbb{C}^*$-action with weights $(1,1,-1,-1)$. The two GIT quotients are then: 
\begin{enumerate}
\item[\customlabel{threefold quotient 1}{$-$}{($-$)}] the resolved conifold $\operatorname{Tot}\!\big(\cO_{\mathbb{P}^1}(-1)^{\oplus 2}\big)$; and
\item[\customlabel{threefold quotient 2}{$+$}{($+$)}] its flop along the zero section $\mathbb{P}^1$.
\end{enumerate}
We again obtain a spherical pair, similarly to above. The categories~$\cW_k$ and~$\cC$ are defined in just the same way. The orthogonals are now generated by~$\cO_{\vec{y}}(-1)$ and~$\cO_{\vec{x}}(1)$ supported respectively on the $(y_1,y_2)$-plane, and~$(x_1,x_2)$-plane.

In this case, however, taking complexified Grothendieck groups does not give the intersection cohomology complex of a local system on the punctured disk $\Delta-0$: the reason is that the flop--flop functor acts trivially at the level of Grothendieck groups. See Example~\ref{example standard flop} for details. 

\begin{remark}This example is the simplest case where Bodzenta--Bondal construct a spherical pair~\cite{BB}: it would be interesting to relate our constructions.\end{remark}

\begin{remark}The examples in this section have natural extensions to higher dimensions, namely to an orbifold flop of local $\mathbb{P}^{\mathrm{odd}}$ (Example~\ref{example orbifold}), and to standard flops (Example~\ref{example standard flop}).\end{remark}

\section{Background}

\subsection{Perverse sheaves}
\label{subsection.perv_sheaves}

This subsection summarizes results on perverse sheaves of vector spaces which will be used later. We first recall a  standard description of the category $\operatorname{Per} (\Delta,0)$ of perverse sheaves on a disk $\Delta \subset \mathbb{C}$ containing $0$, and possibly singular there.

\begin{proposition}\label{proposition.GGM_perverse_description} \cite{Beilinson1984} \cite[Th\'eor\`eme~II.2.3]{GGM} There is an equivalence from $\operatorname{Per} (\Delta,0)$ to the category of diagrams of vector spaces
\[
\begin{tikzpicture}
	\node (zero) at (0,0) {$D_0$};
	\node (plus) at (2,0) {$D_1$};
	\draw[->,transform canvas={yshift=+\arrowsplit}] (plus) to  node[above] {$\scriptstyle u $} (zero);
	\draw[<-,transform canvas={yshift=-\arrowsplit}] (plus) to  node[below] {$\scriptstyle v $} (zero);
\end{tikzpicture}
\]
such that the morphism $m = v u + \Idmatrix$ is an isomorphism.\end{proposition}

The above equivalence may be obtained by letting $D_0$ be vanishing cycles, and $D_1$~be nearby cycles in such a way that the morphism $m$ is realized as a monodromy operator acting on $D_1$.

We now give an alternative description of $\operatorname{Per} (\Delta,0)$, due to \KS{}: it is a special case of their remarkable characterization of perverse sheaves on a complex vector space singular with respect to a real hyperplane arrangement~\cite{KS1}. Their categorification of this description is given in the following Section~\ref{subsection.sph_pair}.

\begin{proposition}\label{proposition.KS_perverse_description} \cite[Section~9]{KS1} There is an equivalence from $\operatorname{Per} (\Delta,0)$ to the category of diagrams of vector spaces
\[
\begin{tikzpicture}
	\node (zero) at (0,0) {$E_0$};
	\node (plus) at (2,0) {$E_+$};
	\node (minus) at (-2,0) {$E_-$};
	\draw[->,transform canvas={yshift=+\arrowsplit}] (minus) to  node[above] {$\scriptstyle f_- $} (zero);
	\draw[<-,transform canvas={yshift=-\arrowsplit}] (minus) to  node[below] {$\scriptstyle g_- $} (zero);
	\draw[->,transform canvas={yshift=+\arrowsplit}] (plus) to  node[above] {$\scriptstyle f_+ $} (zero);
	\draw[<-,transform canvas={yshift=-\arrowsplit}] (plus) to  node[below] {$\scriptstyle g_+ $} (zero);
\end{tikzpicture}
\]
such that
\begin{enumerate}
\item\label{proposition.KS_perverse_description 1} $g_\pm  f_\pm = \Idmatrix$, and
\item\label{proposition.KS_perverse_description 2} $g_\pm  f_\mp$ are isomorphisms.
\end{enumerate}
There exists furthermore an equivalence from the description of Proposition~\ref{proposition.GGM_perverse_description}, by taking diagrams as follows.
\[
\begin{tikzpicture}
	\node (zero) at (0,0) {$D_0 \oplus D_1$};
	\node (plus) at (3,0) {$D_1$};
	\node (minus) at (-3,0) {$D_1$};
	\draw[->,transform canvas={yshift=+\arrowsplit}] (minus) to  node[above] {$\scriptstyle \left(\begin{smallmatrix}u \\ \Idmatrix\end{smallmatrix}\right) $} (zero);
	\draw[<-,transform canvas={yshift=-\arrowsplit}] (minus) to  node[below] {$\scriptstyle \left(\begin{smallmatrix}0 & \Idmatrix\end{smallmatrix}\right) $} (zero);
	\draw[->,transform canvas={yshift=+\arrowsplit}] (plus) to  node[above] {$\scriptstyle \left(\begin{smallmatrix}0 \\ \Idmatrix\end{smallmatrix}\right) $} (zero);
	\draw[<-,transform canvas={yshift=-\arrowsplit}] (plus) to  node[below] {$\scriptstyle \left(\begin{smallmatrix}v & \Idmatrix\end{smallmatrix}\right) $} (zero);
\end{tikzpicture}
\]
In particular, the monodromy actions on $E_-$ and $E_+$ may be taken to be
\[m_-= g_- f_+  g_+ f_- \qquad\text{ and }\qquad m_+= g_+ f_-  g_- f_+\]
respectively.
\end{proposition}
\begin{proof}For the statement on the equivalence with the description of Proposition~\ref{proposition.GGM_perverse_description}, see \cite[Proposition~9.4]{KS1}.
\end{proof}

For $P$ in $\operatorname{Per} (\Delta,0)$ the spaces $E_{\pm}$ and $E_0$ above are given by stalks at~$\pm 1$ and~$0$ respectively of the sheaf $\mathbb{H}^1_K(P)$ of cohomology with support on a skeleton~$K$ as discussed in Section~\ref{section diff eqn}. The maps $g_\pm$ are generalization maps, as described for instance in \cite[Section~1D]{KS1}, and the $f_\pm$ are dual to them.

\begin{remark} By Proposition~\ref{proposition.KS_perverse_description}, $E_0$ is isomorphic to the direct sum of the nearby and vanishing cycles of the perverse sheaf, however this isomorphism is not canonical: see~\cite[Proof of Proposition~9.4]{KS1}. \end{remark}

\begin{proposition}\label{proposition.simple_perverse_sheaves} The simple objects of $\operatorname{Per} (\Delta,0)$ are of the form:
\begin{enumerate}
\item an intersection cohomology complex $\operatorname{IC}(L)$, for  an irreducible local system~$L$ on~$\Delta-0$; or
\item a sheaf $\underline{\field}_0[-1]$.
\end{enumerate}
Under the equivalence of Proposition~\ref{proposition.GGM_perverse_description} these correspond to:
\begin{enumerate}
\item for $F$ a fibre of $L$ with monodromy $m$, and $F^m$ the $m$-invariants, a~diagram
\[
\begin{tikzpicture}
	\node (zero) at (0,0) {$F/F^m$};
	\node (plus) at (2.1,0) {$F;$};
	\draw[->,transform canvas={yshift=+\arrowsplit}] (plus) to  (zero); 
	\draw[<-,transform canvas={yshift=-\arrowsplit}] (plus) to   (zero); 
\end{tikzpicture}
\]

\item a diagram
\[
\begin{tikzpicture}
	\node (zero) at (0,0) {$\phantom{.}\mathbb{C}$};
	\node (plus) at (1.6,0) {$0.$};
	\draw[->,transform canvas={yshift=+\arrowsplit}] (plus) to  (zero); 
	\draw[<-,transform canvas={yshift=-\arrowsplit}] (plus) to   (zero); 
\end{tikzpicture}
\]
\end{enumerate}

\begin{proof}
The simples perverse sheaves on a variety~$V$ are shifted intersection cohomology complexes, namely $\operatorname{IC}(L)[-\operatorname{codim}_V W]$, for $L$ an irreducible local system on an open subset of a closed $W \subseteq V$, by \cite{BBD1982}: for another reference see \cite[Theorem~6.9]{Rietsch}. The first part follows by setting $V=\Delta$, and $W=\Delta$~or~$0$ respectively. The second part then follows from a well-known computation of nearby and vanishing cycles. 
\end{proof}
\end{proposition}

\subsection{Spherical pairs}
\label{subsection.sph_pair}

This subsection gives categorifications of perverse sheaves of vector spaces on a disk, namely spherical pairs and functors, following~\KS{}.

Consider a $\mathbb{C}$-linear triangulated category $\cE$ with admissible subcategories $\cE_\pm$ embedded by
$ \delta_\pm \colon \cE_\pm \rightarrow \cE.$
By definition, the functors $\delta_\pm$ have left and right adjoints, denoted by ${}^*\delta_\pm$~and $\delta_\pm^*$~respectively. Considering the right orthogonals $\cE_\pm^\perp$ we also have embeddings $ \gamma_\pm \colon \cE_\pm^\perp \rightarrow \cE$
with left adjoints ${}^*\gamma_\pm$.\begin{definition}\label{definition sph pair}\cite[Definition~3.5]{KS2} 
The data $(\cE_\pm,\cE)$, along with the embeddings above, gives a \defined{spherical pair} if the following are equivalences:
\renewcommand{\theenumi}{\alph{enumi}}
\begin{enumerate}
\item \label{sph pair functor 1} $ \delta_+^* \compose \delta_- \colon \cE_- \longrightarrow \cE_+ $
\item \label{sph pair functor 2} $ \delta_-^* \compose \delta_+ \colon \cE_+ \longrightarrow \cE_- $
\item \label{sph pair functor 3} $ {}^*\gamma_+ \compose \gamma_- \colon \cE_-^\perp \longrightarrow \cE_+^\perp $
\item \label{sph pair functor 4} $ {}^*\gamma_- \compose \gamma_+ \colon \cE_+^\perp \longrightarrow \cE_-^\perp $
\end{enumerate}
\end{definition}
\noindent In other words, a spherical pair is given by semi-orthogonal decompositions
\begin{equation}\label{equation sod}
\big\langle \cE_-^\perp, \cE^{\vphantom{\perp}}_- \big\rangle = \cE = \big\langle \cE_+^\perp, \cE^{\vphantom{\perp}}_+ \big\rangle,
\end{equation}
where $\cE_\pm$ are admissible subcategories of $\cE$, and the natural compositions (\ref{sph pair functor 1})--(\ref{sph pair functor 4}) of embedding and projection functors are equivalences. Note that \KandS{} write decompositions in the reverse order~\cite[Section~3B]{KS2}.

The following is implicit in~\KS~\cite{KS2}: applying the complexified Grothendieck group construction to a spherical pair yields two perverse sheaves of vector spaces on a disk depending on whether the left- or right-hand components of (\ref{equation sod}) are taken. Write $\K_0(\cE)$ for the Grothendieck group of a triangulated category $\cE$, and $\K(\cE)$ for the \mbox{$\mathbb{C}$-algebra}~$\K_0(\cE) \otimes \mathbb{C}.$

\begin{proposition}\label{proposition.K_theory} Given a spherical pair $\cP = (\cE_\pm,\cE)$ as above, two perverse sheaves $\cP^{\K}$ and ${}^{\K}\cP$ on a disk $\Delta$ may be obtained, given as follows in the notation of Proposition~\ref{proposition.KS_perverse_description}.
\[
\begin{tikzpicture}
	\node (zero) at (0,0) {$\K(\cE)$};
	\node (plus) at (2.5,0) {$\K(\cE_+)$};
	\node (minus) at (-2.5,0) {$\K(\cE_-)$};
	\draw[->,transform canvas={yshift=+\arrowsplit}] (minus) to  node[above] {$\scriptstyle \delta_- $} (zero);
	\draw[<-,transform canvas={yshift=-\arrowsplit}] (minus) to  node[below] {$\scriptstyle \delta_-^* $} (zero);
	\draw[->,transform canvas={yshift=+\arrowsplit}] (plus) to  node[above] {$\scriptstyle \delta_+ $} (zero);
	\draw[<-,transform canvas={yshift=-\arrowsplit}] (plus) to  node[below] {$\scriptstyle \delta_+^* $} (zero);
\end{tikzpicture}
\]
\[
\begin{tikzpicture}
	\node (zero) at (0,0) {$\K(\cE)$};
	\node (plus) at (2.5,0) {$\K(\cE^\perp_+)$};
	\node (minus) at (-2.5,0) {$\K(\cE^\perp_-)$};
	\draw[->,transform canvas={yshift=+\arrowsplit}] (minus) to  node[above] {$\scriptstyle \gamma_- $} (zero);
	\draw[<-,transform canvas={yshift=-\arrowsplit}] (minus) to  node[below] {$\scriptstyle {}^*\gamma_- $} (zero);
	\draw[->,transform canvas={yshift=+\arrowsplit}] (plus) to  node[above] {$\scriptstyle \gamma_+ $} (zero);
	\draw[<-,transform canvas={yshift=-\arrowsplit}] (plus) to  node[below] {$\scriptstyle {}^*\gamma_+ $} (zero);
\end{tikzpicture}
\]
Here the symbols $\delta$ and~$\gamma$ for functors are reused for their images under~$\K(\placeholder)$.
\begin{proof}First note that $\delta_\pm^* \compose \delta_\pm = \Id_{\cE_\pm}$. This follows by the Yoneda lemma, because for instance there are functorial isomorphisms \[\Hom_{\cE_+} \!\big(a, (\delta_+^* \compose \delta_+) (b)\big) = \Hom_{\cE} \!\big(\delta_+ (a), \delta_+ (b)\big) = \Hom_{\cE_+} (a, b), \] where the latter equality holds because $\delta_+$ is an embedding. The conditions (\ref{proposition.KS_perverse_description 1})--(\ref{proposition.KS_perverse_description 2}) of Proposition~\ref{proposition.KS_perverse_description}  for $\cP^{\K}$ may then be verified immediately. The other case ${}^{\K}\cP$ follows by a dual argument. 
\end{proof}
\end{proposition}

Spherical pairs yield spherical functors, as follows. Here we implicitly take enhanced triangulated categories, so that the functorial cones below make sense: see for instance~\cite[Appendix~A]{KS2}. 

\begin{proposition}\label{proposition.spherical_pair_to_functor} \cite[Propositions~3.7, 3.8]{KS2} Given a spherical pair $(\cE_\pm,\cE)$ as above, the functor
\[
\sphFun = {}^*\gamma_+ \compose \, \delta_- \colon \cE_-^{\vphantom{\perp}} \longrightarrow \cE_+^\perp
\]
is spherical in the sense of Anno~\cite{Anno} and Anno--Logvinenko~\cite{AL2}, and we have that:
\begin{enumerate}
\item\label{proposition.spherical_pair_to_functor 1} the twist acts on $\cE_+^\perp$ by 
\[ \twistFun_\sphFun := \operatorname{\sf Cone}\big(\,\sphFun \compose \sphFun^* \xrightarrow{\text{\,\emph{counit}\,\,}} \Id\big) \cong {}^*\gamma_+ \compose \gamma_- \compose {}^*\gamma_- \compose \gamma_+ \]
which is the composition \emph{(\ref{sph pair functor 4})} then~\emph{(\ref{sph pair functor 3})} from Definition~\ref{definition sph pair};
\item\label{proposition.spherical_pair_to_functor 2} the cotwist acts on $\cE_-$ by 
\[ \cotwistFun_\sphFun := \operatorname{\sf Cone}\big(\Id \xrightarrow{\text{\,\,\emph{unit}\,\,}} \sphFun^* \compose \sphFun\,\big)[-1]
\cong \delta^*_- \compose \delta_+ \compose \delta^*_+ \compose \delta_- \]
which is the composition \emph{(\ref{sph pair functor 1})} then~\emph{(\ref{sph pair functor 2})} from Definition~\ref{definition sph pair}.
\end{enumerate}
\end{proposition}

Combining, we also have the following.

\begin{proposition}\label{proposition.spherical_functor_monodromy} 
 Given a spherical pair $(\cE_\pm,\cE)$ as above, with associated spherical functor $\sphFun$, we have that:
\begin{enumerate}
\item\label{proposition.spherical_functor_monodromy 1} the action of $\twistFun_\sphFun$ on $\K(\cE_+^\perp)$ is the monodromy around $0$ of ${}^{\K}\cP$;
\item\label{proposition.spherical_functor_monodromy 2} the action of $\cotwistFun_\sphFun$ on $\K(\cE_-)$ is the monodromy around $0$ of $\cP^{\K}$.
\end{enumerate}

\begin{proof}
This follows from the definitions of $\cP^{\K}$ and ${}^{\K}\cP$ from Proposition~\ref{proposition.K_theory}, their monodromies given in Proposition~\ref{proposition.KS_perverse_description}, and the descriptions of the twist and cotwist in Proposition~\ref{proposition.spherical_pair_to_functor}.
\end{proof}
\end{proposition}

\subsection{Results from GIT}\label{section quotients}

This subsection tersely recalls results from GIT which will be used to study derived categories of GIT quotients. I follow treatments of Halpern-Leistner~\cite{HL}, and Ballard, Favero, and Katzarkov~\cite{BFK}.
\begin{setup}\label{GIT setup} Take a projective-over-affine variety $X$ with an action of a connected reductive group $G$, and a $G$-ample equivariant line bundle~$\cM$.\end{setup}
The above data determines a semistable locus $X^{\mathrm{ss}}(\cM) \subseteq X$: we take the \defined{GIT quotient} to be the quotient stack $X^{\mathrm{ss}}(\cM)/G$. Note that this may be a Deligne--Mumford stack: see Section~\ref{section simple} and Example~\ref{example orbifold} for simple examples. After certain choices, we obtain a stratification of $X - X^{\mathrm{ss}}(\cM)$ by strata $S^i$. I~now give the properties of this stratification needed in what follows, leaving further details to the references.

To each stratum $S^i$ is associated, by construction, a one-parameter subgroup $\lambda^i$, and the stratum contains an associated open subvariety of the $\lambda^i$-fixed locus, denoted by $Z^i$. Consider a single stratum $S$, dropping indices~$i$ for simplicity. It contains furthermore a subvariety $Y \subseteq X$ flowing to $Z$ under $\lambda$, referred to as a `blade', which will be important: writing the elements of $\lambda$ as $\lambda_t$ for $t \in \mathbb{C}^*,$ this blade is given as follows.
\[
Y = \Big\{\, x \mathrel{\Big|} \underset{t\to 0}{\operatorname{lim}} \, (\lambda_t\cdot x)\text{ lies in }Z \,\Big\} \subset X.
\]
Note that $Z \subseteq Y$, and that $Y$ has a natural projection to $Z$. Notate morphism as below.
\[
\begin{tikzpicture}
	\node (Z) at (0,0) {$\phantom{.}Z$};
	\node (S) at (2,0) {$Y \subseteq S$};
	\node (X) at (4,0) {$X$};
	\draw[right hook->,transform canvas={yshift=+\arrowsplit}] (Z) to  node[above] {$\scriptstyle \sigma$} (S);
	\draw[<<-,transform canvas={yshift=-\arrowsplit}] (Z) to node [below]  {$\scriptstyle \pi$} (S);
	\draw[right hook->] (S) to  node[above] {$\scriptstyle j$} (X);
\end{tikzpicture}
\]
These morphisms may be made equivariant, in the following manner. Write $L$ for the subgroup of elements of $G$ which commute with $ \lambda_t$ for all $t$, and put
\[
P = \Big\{\, g \mathrel{\Big|} \underset{t\to 0}{\operatorname{lim}} \, (\lambda_t \cdot g \cdot \lambda_t^{-1}) \text{ lies in }L \,\Big\} \subseteq G.
\]
Observe that $L \subseteq P$, and that there is a map $P \to L$ given by
\[ g \mapsto \underset{t\to 0}{\operatorname{lim}}\left(\lambda_t  \cdot g \cdot \lambda_t^{-1}\right). \]
There are then morphisms of quotient stacks as follows, reusing the notation above.
\[
\begin{tikzpicture}
	\node (Z) at (0,0) {$\phantom{.}Z/L$};
	\node (S1) at (2,0) {$Y/P$};
	\node (S2) at (4,0) {$S/G$};
	\node (X) at (6,0) {$X/G.$};
	\draw[right hook->,transform canvas={yshift=+\arrowsplit}] (Z) to  node[above] {$\scriptstyle \sigma$} (S1);
	\draw[<<-,transform canvas={yshift=-\arrowsplit}] (Z) to node [below]  {$\scriptstyle \pi$} (S1);
	\draw[right hook->] (S2) to  node[above] {$\scriptstyle j$} (X);
	\draw[->] (S1) to  node[above] {$\scriptstyle \sim$} (S2);
\end{tikzpicture}
\]
The middle morphism is induced by the inclusion of $Y$ into $S$, and is an equivalence: see for instance \cite[Section~2.1, Property~(S2)]{HL}.

\begin{remark}If $G$ is abelian, the blade $Y$ coincides with the stratum $S$, and the subgroups $L$ and $P$ are equal to $G$.
\end{remark}

The following properties of the stratification will be used later.

\begin{proposition}\label{proposition stratification props}
If $X$ is smooth in a $G$-equivariant neighbourhood of $Z$:
\begin{enumerate}
\item\label{proposition stratification props 1}
the sheaf $\cN^\vee_Y X$ restricted to $Z$ has strictly positive $\lambda$-weights;
\item\label{proposition stratification props 2a}
the variety $X$ is smooth in a $G$-equivariant neighbourhood of $Y$\!;
\item\label{proposition stratification props 2b}
the subvarieties $Y$\! and $Z$ are smooth;
\item\label{proposition stratification props 4}
the morphism $\pi\colon Y \to Z$ is a locally trivial bundle of affine spaces;
\item\label{proposition stratification props 3}
the stratum $S$ is regularly embedded in $X$.
\end{enumerate}
\end{proposition}
\begin{proof}The definition of the blade $Y$ yields (\ref{proposition stratification props 1}). For the rest, we follow the approach of \cite[Lemma~2.7]{HL}. To see~(\ref{proposition stratification props 2a}), note that any $G$-invariant neighbourhood of $Z$ contains $Y$. Then (\ref{proposition stratification props 2b}) and~(\ref{proposition stratification props 4}) follow from the Bia\l{}ynicki-Birula decomposition~\cite[Theorem~4.1]{BBirula} applied to the action of $\lambda$ on this neighbourhood, and~(\ref{proposition stratification props 3}) is given by an argument of Kirwan~\cite[Theorem~13.5]{Kirwan}.
\end{proof}

\subsection{Derived categories}\label{subsection.decomp} 

This subsection follows Halpern-Leistner~\cite{HL}, Halpern-Leistner and Shipman~\cite{HLShipman}, and Ballard, Favero, and Katzarkov~\cite{BFK}. Assume for simplicity that the stratification of the previous subsection consists of a single stratum with associated subvariety $Z$, and that $X$ is a smooth in a $G$-equivariant neighbourhood of $Z$. We define windows in $\D(X/G)$ which will be equivalent to the derived category $\D(X^{\mathrm{ss}}/G)$ of the GIT quotient, by measuring weights on $Z$.

As notation, for $F^{\bullet} \in \D(X/G)$, write $\operatorname{wt}_\lambda F^{\bullet}$ for the set of $\lambda$-weights which appear in some cohomology sheaf $\operatorname{\mathcal{H}}^k (\sigma^* j^* F^{\bullet})$ on~$Z$.
Then given an integer~$w$, a window $\mathcal{G}^w$ is defined as the full subcategory of $\D(X/G)$ with objects
\begin{align*}
\mathcal{G}^w &= \big\{\, F^{\bullet} \in \D(X/G) \mathrel{\big|} \operatorname{wt}_\lambda F^{\bullet} \subseteq [w,w+\eta) \,\big\},
\end{align*}
where $\eta$ is defined as follows.
\begin{definition}\label{definition window width} The \defined{window width} $\eta$ is the $\lambda$-weight of $\operatorname{det} \mathcal{N}^{\vee}_S X$ on $Z$.\end{definition}

The category $\mathcal{G}^w$ is equivalent to $\D(X^{\mathrm{ss}}/G)$: indeed we have the following more refined result.
\begin{theorem}\label{theorem.splittings_of_res}\cite{HL,BFK}
Define a full subcategory \begin{align*}
\mathcal{C}^w &= \big\{\, F^{\bullet} \in \D(X/G) \mathrel{\big|} \operatorname{wt}_\lambda F^{\bullet} \subseteq [w,w+\eta] \,\big\}
\end{align*}
of $\D(X/G)$ and consider the restriction functor to the GIT~quotient
\[ \operatorname{res} \colon \mathcal{C}^w \subset \D(X/G) \longrightarrow \D(X^{\mathrm{ss}}/G). \]
This  is an equivalence on the subcategories $\cG^w$ and $\cG^{w+1}$, and has adjoints
\[ \operatorname{res}^* \colon \D(X^{\mathrm{ss}}/G) \overset{\sim}{\longrightarrow} \mathcal{G}^w \subset \mathcal{C}^w \quad \text{and} \quad
{}^*\!\operatorname{res} \colon \D(X^{\mathrm{ss}}/G) \overset{\sim}{\longrightarrow} \mathcal{G}^{w+1} \subset \mathcal{C}^w.
\]

\end{theorem}
\begin{proof}See~\cite[Lemma 2.4]{HLShipman} for this statement.
\end{proof}

Note that the adjoint $\operatorname{res}^*$ is an embedding $ \D(X^{\mathrm{ss}}/G) \to \mathcal{C}^w$. The following lemma and proposition give furthermore a semi-orthogonal decomposition of $\mathcal{C}^w$. Write $\D(Z/L)^w$ for the full subcategory of $\D(Z/L)$ whose objects have $\lambda$-weight~$w$.

\begin{lemma}\label{proposition.iota_adjoints}\cite{HLShipman} For a given $\lambda$-weight $w$, the functor
\[  \iota = j_* \pi^* ( \placeholder ) \colon \D(Z/L)^w \longrightarrow \mathcal{C}^w \]
is an embedding, with left and right adjoints
\[
{}^*\iota = \big(\sigma^* j^* (\placeholder) \big) ^w \quad \text{and} \quad
\iota^{*} = \big(\sigma^* j^* (\placeholder) \big) ^{w+\eta} \otimes \operatorname{det} \mathcal{N}_S X|_Z [-\operatorname{codim}_X S],
\]
where the superscripts denote projection to weight subcategories of $\D(Z/L)$. 
\end{lemma}
\begin{proof}
This is~\cite[Lemma 2.3]{HLShipman} with a shift~$[-\operatorname{codim}_X S]$ added in the statement coming from Grothendieck duality for the morphism~$j$.
\end{proof}

\begin{proposition}\label{proposition sod}\cite{HL,BFK} There is a semi-orthogonal decomposition
\[
\mathcal{C}^w = \big\langle\! \D(X^{\mathrm{ss}}/G), \,\D(Z/L)^w \big\rangle
\]
with embedding functors $\operatorname{res}^*$ and $\iota$.
\end{proposition}
\begin{proof}
This follows from \cite{HL,BFK} and may be found in this form in \cite[\mbox{Section~2, (3)}]{HLShipman}.\qedhere\end{proof}

\subsection{Variation of GIT}
\label{subsection vgit}
This subsection sets up the necessary technology to study GIT wall crossings. Consider a GIT problem $X/G$ as in Setup~\ref{GIT setup}. The~results of the previous sections continue to hold if we replace the \mbox{$G$-ample} equivariant line bundle~$\cM$ used there with an element of $\operatorname{NS}^G(X) \otimes \mathbb{R}$, namely the scalar extension to $\mathbb{R}$ of the equivariant Neron--Severi group of $X$. Elements of this group will henceforth be referred to as linearizations, so that we may consider continuous variation of linearization: see \cite[Section~3.2]{BFK} for a fuller treatment. Take then linearizations $\linzn_0$ and $\linzn'$, and let \[\linzn_t = \linzn_0 + t \linzn'\]  for each $t\in[-\epsilon, \epsilon]$ with $\epsilon$ small such that:

\begin{enumerate}
\item[\customlabel{wall crossing 0}{\,$0$\,}{(\,$0$\,)}] there exist strict semistables for $\linzn_0$;
\item[\customlabel{wall crossing -}{$-$}{($-$)}] $X^{\mathrm{ss}}(\linzn_t)$ is constant for $t \in [-\epsilon,0)$, with no strict semistables;
\item[\customlabel{wall crossing +}{$+$}{($+$)}] $X^{\mathrm{ss}}(\linzn_t)$ is constant for $t \in (0, \epsilon]$, with no strict semistables.
\end{enumerate}
In this situation write $\linzn_\pm = \linzn_{\pm\epsilon}$, and say that the transition between $\linzn_\pm$ is a \defined{wall crossing}, and $\linzn_0$ is a linearization \defined{on the wall}. For brevity we denote the GIT quotients corresponding to either side of the wall as follows.
\[
X \gitQuotStab{\pm} G = X^{\mathrm{ss}}(\linzn_\pm) / G.
\]
The semistable locus for $\linzn_0$ may be related to the loci for $\linzn_\pm$ by the formulae
\[
X^{\mathrm{ss}}(\linzn_0) = X^{\mathrm{ss}}(\linzn_\pm) \cup \bigcup_{i \in I_\pm} S^i_\pm
\]
where the $I_\pm$ are indexing sets for unstable strata under the respective linearizations $\linzn_\pm$~\cite[Section~4]{DH}.
We now recall a notion of Halpern-Leistner~\cite[Definition~4.4]{HL}: a similar situation is considered by Ballard, Favero, and Katzarkov in~\cite[Condition~4.3.1]{BFK}.
\begin{definition}\label{definition balanced} A wall crossing is \defined{balanced} if there is an identification $I_- = I_+$ under which one-parameter subgroups~$\lambda_\pm$, associated~$L_\pm$, and subvarieties~$Z_\pm$ correspond as follows:
\[ \lambda_-^i = (\lambda_+^i)^{-1}, \quad L_-^i = L_+^i, \quad \text{and} \quad Z_-^i/L_-^i = Z_+^i/L_+^i . \] 
\end{definition}
I make the following further restriction, for use in what  follows.
\begin{definition}\label{definition simple balanced} A wall crossing is \defined{simple balanced} if it is balanced, and furthermore $|I_-|=|I_+|=1$.
\end{definition}
In this case, we drop indices, and write simply $Z$ and $L$ for the identified subvarieties~$Z_\pm$, and subgroups $L_\pm$.

\section{Spherical pairs}
\label{section.balanced}

Assume given a simple balanced wall crossing as in Definition~\ref{definition simple balanced} for a GIT problem $X/G$ as in Setup~\ref{GIT setup}. In this section we construct a spherical pair in the sense of Definition~\ref{definition sph pair}, and describe the equivalences (\ref{sph pair functor 1})--(\ref{sph pair functor 4}) which appear there. Recall that we take a wall crossing between linearizations~$\linzn_\pm$, and a linearization~$\linzn_0$ on the wall, with an associated subvariety~$Z$ and subgroup~$L$, and that the window width $\eta_\pm$ is defined to be the $\lambda_\pm$-weight of $\det \cN^\vee_{S_\pm} X$~on~$Z$. The GIT quotients corresponding to either side of the wall are denoted by $X \gitQuotStab{\pm} G$.

\begin{theorem}\label{theorem.sph_pair} Assume that:

\renewcommand{\theenumi}{\roman{enumi}}
\begin{enumerate}
\item\label{theorem assumption 1} the variety $X$ is smooth in a $G$-equivariant neighbourhood of $Z$;
\item\label{theorem assumption 2} the canonical sheaf $\omega_X$ has $\lambda_\pm$-weight zero on $Z$; and
\item\label{theorem assumption 3} the group $G$ is abelian.
\end{enumerate}
It follows that $\eta_+=\eta_-$, and we write $\eta$ for their common value. Fix an integer $w$, and consider the category
\[
\mathcal{C} = \left\{ \begin{array}{c|c} F^{\bullet} \in \D\!\big(X^{\mathrm{ss}}(\linzn_0)/G\,\big) & \operatorname{wt}_{\lambda_-} F^{\bullet} \subseteq [w,w+\eta] \end{array} \right\}.
\]
Then there exist embeddings
\begin{enumerate}
\item[\customlabel{embedding 1}{$-$}{($-$)}] $\iota_{-\vphantom{+}} \colon \D(Z/L)_-^{\vphantom{\eta} w} \longrightarrow \cC$
\item[\customlabel{embedding 2}{$+$}{($+$)}] $\iota_+ \colon \D(Z/L)_+^{-w-\eta} \longrightarrow \cC$
\end{enumerate}
where the superscripts denote $\lambda_-$-weight and $\lambda_+$-weight subcategories respectively, such that
\renewcommand{\theenumi}{\arabic{enumi}}
\begin{enumerate}
\item\label{theorem.sph_pair 1} the data $(\operatorname{Im} \iota_\pm, \cC)$ gives a spherical pair $\cP$,
\end{enumerate}
and furthermore:
\begin{enumerate}
\addtocounter{enumi}{1}
\item\label{theorem.sph_pair 2} the right orthogonals $(\operatorname{Im} \iota_\pm)^\perp$ are equivalent to $\D(X \gitQuotStab{\pm} G)$ via $\operatorname{res}_\pm$;
\item\label{theorem.sph_pair 3} the compositions \emph{(\ref{sph pair functor 1})}--\emph{(\ref{sph pair functor 4})} from Definition~\ref{definition sph pair} are given by
\renewcommand{\theenumi}{\alph{enumi}}
\begin{enumerate}
\item\label{composition 1} $ \placeholder \otimes \det \cN_{S_+} X|_Z [ -\operatorname{codim}_X {S_+} ] \colon \D(Z/L)_{-\vphantom{+}}^{w\vphantom{\eta}} \longrightarrow \D(Z/L)_+^{-w-\eta}$
\item\label{composition 2} $ \placeholder \otimes \det \cN_{S_-} X|_Z [ -\operatorname{codim}_X {S_-} ] \colon \D(Z/L)_+^{-w-\eta} \longrightarrow \D(Z/L)_{-\vphantom{+}}^{w\vphantom{\eta}} $
\item\label{composition 3} $\operatorname{res}_+ \compose \operatorname{res}_{-\vphantom{+}}^* \colon \D(X \gitQuotStab{-} G) \longrightarrow \D(X \gitQuotStab{+} G)$
\item\label{composition 4} $\operatorname{res}_- \compose\operatorname{res}_+^* \colon \D(X \gitQuotStab{+} G) \longrightarrow \D(X \gitQuotStab{-} G)$
\end{enumerate}
with \emph{(\ref{sph pair functor 3})} and \emph{(\ref{sph pair functor 4})} factoring via window subcategories $\cG_-^w$ and $\cG_+^{-w-\eta}$ of $\D\!\big(X^{\mathrm{ss}}(\linzn_0)/G\,\big)$\! respectively.
\end{enumerate}
\end{theorem}

\begin{remark} The equivalences (\ref{sph pair functor 3}) and (\ref{sph pair functor 4}) of the theorem are those obtained by Halpern-Leistner~\cite[Section~4.1]{HL} and Ballard--Favero--Katzarkov~\cite{BFK}.
\end{remark}

\begin{remark} The weight assumption (\ref{theorem assumption 2}) of the theorem is certainly satisfied, for instance, if $X$ is $G$-equivariantly Calabi--Yau.
\end{remark}

\begin{proof}[Proof of Theorem~\ref{theorem.sph_pair}]
Assumption~(\ref{theorem assumption 1}) implies that $X$ is smooth in a neighbourhood of $Y_\pm$, and that $Y_\pm$ and $Z$ are smooth, using Propositions~\ref{proposition stratification props}(\ref{proposition stratification props 2a}) and~\ref{proposition stratification props}(\ref{proposition stratification props 2b}). The claim that $\eta_+=\eta_-$ may then be deduced from assumption~(\ref{theorem assumption 2}) as follows, as in \cite[proof of Proposition~4.5]{HL}. First note that $\Omega_{X}$ on $Z$ is split by sign of $\lambda_\pm$-weights as 
\begin{equation}\label{splitting of normal bundle Z}
\cN^{\vee}_{Y_-} \!X|_{Z} \,\oplus\, \Omega_Z \,\oplus\, \cN^{\vee}_{Y_+} \!X|_{Z}.
\end{equation}
An example is sketched in Figure~\ref{figure blades}. Then the claim is shown by taking determinants, noting that because $G$ is abelian  the strata $S_\pm$ are identified with the $Y_\pm$, and recalling that $\eta_\pm$ is the $\lambda_\pm$-weight of $\det \cN^\vee_{S_\pm} \!X$ on $Z$.

We next establish the embeddings (\ref{embedding 1}) and~(\ref{embedding 2}). Consider the semistable substack for the linearization $\cM_0$ as follows \[ X^\circ = X^{\mathrm{ss}}(\linzn_0).\] We may replace $X/G$ with its open substack $X^\circ\!/G$, noting that assumptions (\ref{theorem assumption 1})--(\ref{theorem assumption 3}) continue to hold. The GIT stratifications associated to the linearizations $\cM_\pm$ then consist of single unstable strata $S_\pm$, so that we are in the setting of Section~\ref{subsection.decomp}. The equalities \[\cC = \cC^w_- = \cC^{-w-\eta}_+\] are then immediate from the definitions in Theorem~\ref{theorem.splittings_of_res}, using the fact that $\lambda_+ = \lambda_-^{-1}$. Hence the functors $\iota_\pm$ of (\ref{embedding 1}) and~(\ref{embedding 2}) may be defined using Lemma~\ref{proposition.iota_adjoints}, and are embeddings as claimed.

\begin{figure}[h]
\def\rotation{70}
\begin{tikzpicture}[node distance=1cm, auto, line width=0.5pt]
\draw[mid arrow] (0,0) to [bend left=20] (0+\rotation:1.5);
\draw[mid arrow] (0,0) to [bend left=20] (180+\rotation:1.5);
\draw[mid arrow] (90+\rotation:1.5) to [bend left=20] (0,0);
\draw[mid arrow] (270+\rotation:1.5) to [bend left=20] (0,0);
\draw[fill=black] (0,0) circle (.2ex);
\node at (1.5,-0.2) {$Y_+$};
\node at (0.8,1.1) {$Y_-$};
\node at (-0.3,-0.2) {$Z$};
\end{tikzpicture}
\caption{Sketch of blades in $X$, with action of $\lambda_-$.}
\label{figure blades}
\end{figure}
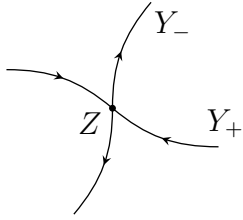

We prove (\ref{theorem.sph_pair 2}) and (\ref{theorem.sph_pair 3}) before concluding (\ref{theorem.sph_pair 1}). Recall that for a spherical pair we require that $\iota_\pm$  have left and right adjoints: these are provided by Lemma~\ref{proposition.iota_adjoints}. By the semi-orthogonal decomposition of Proposition~\ref{proposition sod}, the right orthogonals $(\operatorname{Im} \iota_\pm)^\perp$ are the images of the following embeddings from Theorem~\ref{theorem.splittings_of_res}.
\begin{enumerate}
\item[\customlabel{dual embedding 1}{$-$}{($-$)}] $\operatorname{res}_-^* \colon \D(X \gitQuotStab{-} G) \overset{\sim}{\longrightarrow} \cG_-^w \subset \cC$
\item[\customlabel{dual embedding 2}{$+$}{($+$)}] $\operatorname{res}_+^* \colon \D(X \gitQuotStab{+} G) \overset{\sim}{\longrightarrow} \cG_+^{-w-\eta} \subset \cC$
\end{enumerate}
To satisfy the definition of a spherical pair we require that these embeddings have left adjoints: these are given by $\operatorname{res}_\pm$, and we may deduce~(\ref{theorem.sph_pair 2}).

For~(\ref{theorem.sph_pair 3}), it remains to describe the compositions of functors (\ref{sph pair functor 1})--(\ref{sph pair functor 4}) from Definition~\ref{definition sph pair}. After checking that they are equivalences, we may then conclude that $(\operatorname{Im} \iota_\pm, \cC)$ is a spherical pair, giving (\ref{theorem.sph_pair 1}). Consider first the compositions (\ref{sph pair functor 3}) and (\ref{sph pair functor 4}) given by
$\operatorname{res}_+ \compose \operatorname{res}_{-\vphantom{+}}^*$ and $\operatorname{res}_- \compose\operatorname{res}_+^*$ respectively. The functors $\operatorname{res}_-^*$ and~$\operatorname{res}_+^*$ go to
\[\cG_{\vphantom{+}-}^{\vphantom{\eta}w} = \cG_+^{-w-\eta+1} \qquad \text{and} \qquad \cG_+^{-w-\eta} =\cG_{\vphantom{+}-}^{\vphantom{\eta}w+1} \]  so the compositions are equivalences because $\operatorname{res}_+$ and~$\operatorname{res}_-$ restrict to equivalences on the latter subcategories, by Theorem~\ref{theorem.splittings_of_res}.

Finally, consider the compositions (\ref{sph pair functor 1}) and (\ref{sph pair functor 2}) given as follows.
\renewcommand{\theenumi}{\alph{enumi}}
\begin{enumerate}
\item\label{composition calc 1} $ \iota_+^* \,\compose\, \iota_{-\vphantom{+}} \colon \D(Z/L)_-^w \longrightarrow \D(Z/L)_+^{-w-\eta} $
\item\label{composition calc 2} $ \iota_-^* \,\compose\, \iota_+ \colon \D(Z/L)_+^{-w-\eta} \longrightarrow \D(Z/L)_{-\vphantom{+}}^w  $
\end{enumerate}
\renewcommand{\theenumi}{\roman{enumi}}
These expressions may be expanded using Lemma~\ref{proposition.iota_adjoints}. Recall that morphisms are denoted as follows.
\begin{equation}
\label{equation intersection of Ss Ys}
\begin{tikzpicture}[baseline=-3pt]
	\node (X) at (1.3,0) {$X/G$};
	\node (minusS) at (0,1.3) {$S_-/G$};
	\node (plusS) at (0,-1.3) {$S_+/G$};
	\node (minus) at (-2,1.3) {$Y_-/P_-$};
	\node (plus) at (-2,-1.3) {$Y_+/P_+$};
	\node (Z) at (-3.3,0) {$Z/L$};
	\draw[right hook->] (minusS) to node[above right] {$\scriptstyle j_-$} (X);
	\draw[left hook->] (plusS) to node[below right] {$\scriptstyle j_+$} (X);
	\draw[right hook->,transform canvas={yshift=+\arrowsplit,xshift=-\arrowsplit}] (Z) to node[above left] {$\scriptstyle \sigma_-$} (minus);
	\draw[->>,transform canvas={yshift=-\arrowsplit,xshift=+\arrowsplit}] (minus) to node[below right] {$\scriptstyle \pi_-$} (Z);
	\draw[left hook->,transform canvas={yshift=-\arrowsplit,xshift=-\arrowsplit}] (Z) to node[below left] {$\scriptstyle \sigma_+$} (plus);
	\draw[->>,transform canvas={yshift=+\arrowsplit,xshift=+\arrowsplit}] (plus) to node[above right] {$\scriptstyle \pi_+$} (Z);
	\draw[->] (plus) to node[below,sloped] {$\scriptstyle \sim$} (plusS);
	\draw[->] (minus) to node[above,sloped] {$\scriptstyle \sim$} (minusS);
\end{tikzpicture}
\end{equation}
The composition (\ref{sph pair functor 2}) is given by:
\begin{itemize}
\item applying $ \big(\sigma_-^* j_-^* j^{\phantom*}_{+*}  \pi_+^* (\placeholder)\big)^{w+\eta} $
where the superscript denotes restriction to the $\lambda_-$-weight $w+\eta$ subcategory of $\D(Z/L)$; then
\item tensoring by $ \operatorname{det} \mathcal{N}^{\phantom{\vee}}_{S_-} \!X|_Z [-\operatorname{codim}_X S_-] $.
\end{itemize}
To show that (\ref{sph pair functor 2}) is an equivalence with the form claimed, it therefore suffices to prove Lemma~\ref{lemma compute compositions} below. The result for (\ref{sph pair functor 1}) follows similarly, and we are done.
\end{proof}

\begin{lemma}\label{lemma compute compositions}
The following functors appearing in the proof of Theorem~\ref{theorem.sph_pair} are isomorphic to the identity, after identifying~$\D(Z/L)_{-\vphantom{+}}^w$ and~$\D(Z/L)_+^{-w}.$
\renewcommand{\theenumi}{\alph{enumi}}
\begin{enumerate}
\item\label{lemma compute compositions 1} $\big(\sigma_+^*  j_+^*  j^{\phantom*}_{-*}  \pi_-^* (\placeholder)\big)^{-w}\colon  \D(Z/L)_{-\vphantom{+}}^{\vphantom{\eta}w} \longrightarrow \D(Z/L)_+^{\vphantom{\eta}-w}$
\item\label{lemma compute compositions 2} $\big(\sigma_-^*  j_-^*  j^{\phantom*}_{+*}  \pi_+^* (\placeholder)\big)^{w+\eta}\colon  \D(Z/L)_+^{-w-\eta} \longrightarrow \D(Z/L)_{-\vphantom{+}}^{w+\eta}$
\end{enumerate}
Here superscripts are used as before to denote restriction to $\lambda_-$-weight and $\lambda_+$-weight subcategories, and morphisms are as in \emph{(\ref{equation intersection of Ss Ys})} above.
\end{lemma}

\begin{remark}I give a proof of this lemma using that the blades $Y_\pm$ intersect transversally in $Z$ for $G$ abelian. For $G$ non-abelian, the above equality $\eta_+ = \eta_-$ continues to hold because of a formula \cite[(4)]{HL} of Halpern-Leistner, but I could not find a way to prove the lemma, or a counterexample: this matter did not seem worth pursuing further, as non-abelian examples may, at least in principle, be reduced to abelian ones by a trick of Thaddeus~\cite[Section 3.1]{Thaddeus}.\end{remark}

\begin{proof}[Proof of Lemma~\ref{lemma compute compositions}]
Recall that, for $G$ abelian, the strata~$S_\pm$ are identified with the blades~$Y_\pm$, and $L=P_\pm=G$, so that~(\ref{equation intersection of Ss Ys}) simplifies to the following diagram of $G$-equivariant morphisms.
\begin{equation}\label{equation intersection of Ys}
\begin{tikzpicture}[baseline=-3pt]
	\node (X) at (1.3,0) {$X$};
	\node (minus) at (0,1.3) {$Y_-$};
	\node (plus) at (0,-1.3) {$Y_+$};
	\node (Z) at (-1.3,0) {$Z$};
	
	\draw[right hook->] (minus) to node[above right] {$\scriptstyle j_-$} (X);
	\draw[left hook->] (plus) to node[below right] {$\scriptstyle j_+$} (X);

	\draw[right hook->,transform canvas={yshift=+\arrowsplit,xshift=-\arrowsplit}] (Z) to node[above left] {$\scriptstyle \sigma_-$} (minus);
	\draw[->>,transform canvas={yshift=-\arrowsplit,xshift=+\arrowsplit}] (minus) to node[below right] {$\scriptstyle \pi_-$} (Z);

	\draw[left hook->,transform canvas={yshift=-\arrowsplit,xshift=-\arrowsplit}] (Z) to node[below left] {$\scriptstyle \sigma_+$} (plus);
	\draw[->>,transform canvas={yshift=+\arrowsplit,xshift=+\arrowsplit}] (plus) to node[above right] {$\scriptstyle \pi_+$} (Z);
\end{tikzpicture}
\end{equation}
I claim that $G$-equivariant base change around this Cartesian square holds, namely a relation between equivariant derived functors
\begin{equation}\label{equation base change}
j_-^* j^{\phantom*}_{+*} \cong \sigma^{\phantom*}_{-*}  \sigma_+^*.
\end{equation}
Consider the non-equivariant base change first. This follows by the argument of \cite[Appendix~A]{Addington} for instance. By the splitting (\ref{splitting of normal bundle Z}) of $\Omega_X$ on $Z$ we know that $Z$ is a transverse intersection of the $Y_\pm$, in particular it has the expected dimension. It then suffices to know that $X$ is smooth in a neighbourhood of $Y_{\pm}$, and that $Y_{\pm}$ and $Z$ are smooth: these facts follow from  assumption~(\ref{theorem assumption 1})  of Theorem~\ref{theorem.sph_pair} that $X$ is smooth in a $G$-equivariant neighbourhood of $Z$, by Propositions~\ref{proposition stratification props}(\ref{proposition stratification props 2a}) and~\ref{proposition stratification props}(\ref{proposition stratification props 2b}).

We claim furthermore a base change relation~(\ref{equation base change}) between equivariant derived functors. For this we may use the framework of Bernstein--Lunts~\cite{BL}. I~give the key points of the argument. Recall that to specify an object~$F$ of the equivariant derived category $\D(X/G)$~\cite[Section~2.4.5]{BL} we proceed as follows.
\begin{itemize}
\item Consider smooth resolutions $\phi\colon X' \to X$ where, by definition, $X'$ is a free $G$-space so that the quotient map $X' \to X'\!/G$ is a locally trivial fibration with fibre $G$~\cite[Definition~2.1.1(b)]{BL}, with $\phi$ a $G$-equivariant smooth map~\cite[Section~1.7]{BL}. 
\item Take an object $F_\phi$ of the ordinary derived category of the free quotient $X'\!/G$ for each such resolution.
\end{itemize}
Now pulling back (\ref{equation intersection of Ys}) along $\phi$ gives a $G$-equivariant square of resolutions, and we may then take quotients by free $G$-actions to get a Cartesian square as follows.
\begin{equation}\label{equation intersection of Ys quot}
\begin{tikzpicture}[baseline=-3pt]
	\node (X) at (1.3,0) {$X'\!/G$};
	\node (minus) at (0,1.3) {$Y'_-/G$};
	\node (plus) at (0,-1.3) {$Y'_+/G$};
	\node (Z) at (-1.3,0) {$Z'\!/G$};
	
	\draw[right hook->] (minus) to node[above right] {$\scriptstyle j_-$} (X);
	\draw[left hook->] (plus) to node[below right] {$\scriptstyle j_+$} (X);

	\draw[right hook->] (Z) to node[above left] {$\scriptstyle \sigma_-$} (minus);
	\draw[left hook->] (Z) to node[below left] {$\scriptstyle \sigma_+$} (plus);
\end{tikzpicture}
\end{equation}
To prove equivariant base change~(\ref{equation base change}) following the method of~\cite[Sections~3.3 and~3.4]{BL} we need to show, in particular, base change for the square (\ref{equation intersection of Ys quot}). For this we argue as in the non-equivariant case, using the following observations. The square~(\ref{equation intersection of Ys quot}) is an intersection of the expected dimension because: the resolution maps ($X' \to X$ and so on) have locally constant, and equal, fibre dimensions; and the quotient maps ($X' \to X'\!/G$ and so on) have constant fibre dimension~$\operatorname{dim}(G)$. The smoothness required follows from the smoothness shown in the non-equivariant setting, using that the quotient maps are locally trivial fibrations, and noting furthermore that the neighbourhood of~$Y_\pm$ may be taken to be $G$-equivariant.

Taking the functor (\ref{lemma compute compositions 2}) and applying base change~(\ref{equation base change}), we are led to consider
\begin{equation*}\label{equation.apply_base_change-}
\sigma_-^* \sigma^{\phantom*}_{-*}   \sigma_+^*  \pi_+^*
\end{equation*}
and because $\pi_+ \sigma_+ = \Id_Z$, this further reduces to
\begin{equation}\label{equation.reduce_to_push_pull-}
\sigma_-^* \sigma^{\phantom*}_{-*}.
\end{equation}

The normal bundle sequence associated to the inclusion $\sigma_-$ is split because $\pi_- \sigma_- = \Id_Z$: it follows, as in Arinkin--C\u{a}ld\u{a}\-raru~\cite{Arinkin-Caldararu}, that~\eqref{equation.reduce_to_push_pull-} is isomorphic to tensoring by the direct sum of 
\[ \bigwedge{}^{\!\!\! k} \cN^\dual_Z Y_- [k]  \]
for integers $k \in [0, \operatorname{codim}_{Y_-} Z].$ The  functor (\ref{lemma compute compositions 2}) therefore acts by tensoring with the $\lambda_-$-weight $0$ subobject of this. The $Y_\pm$ intersect transversally in $Z$, so we have $\cN^\dual_Z Y_- \cong \cN^\dual_{Y_+} \!X|_Z$. Proposition~\ref{proposition stratification props}(\ref{proposition stratification props 1}) gives that the $\lambda_+$-weights of the latter are strictly positive, thence the $\lambda_-$-weights strictly negative, so this subobject is just
\[ \bigwedge{}^{\!\!\! 0} \cN^\dual_Z Y_- = \cO_Z,\]
and therefore (\ref{lemma compute compositions 2}) is the identity as claimed. The proof for (\ref{lemma compute compositions 1}) is similar. 
\end{proof}

Combining the theorem with Proposition~\ref{proposition.spherical_pair_to_functor}, we immediately have the following, recovering a result of  Halpern-Leistner and Shipman~\cite[Section~3.2]{HLShipman}.

\begin{corollary}\label{corollary sph functor} In the setting of Theorem~\ref{theorem.sph_pair}, there is a spherical functor 
\[
\sphFun =  \operatorname{res}_+ \circ \, \iota_{-} \colon \cD  \longrightarrow \D(X \gitQuotStab{+} G)
\]
where $\cD$ denotes the category $\D(Z/L)_-^w$.
\end{corollary}

Finally, note that there is a dual notion of spherical pair, given by swapping the roles of the functors in the definition: therefore, as an immediate corollary of the proof of Theorem~\ref{theorem.sph_pair}, we have the following.

\begin{corollary}\label{corollary dual pair} In the setting of Theorem~\ref{theorem.sph_pair}, the data $\big(\!\D(X \gitQuotStab{\pm} G), \cC\big)$ with embeddings 
\[
\operatorname{res}_\pm^* \colon \D(X \gitQuotStab{\pm} G) \longrightarrow \mathcal{C}
\]
gives a \defined{dual spherical pair}, in the sense that it satisfies Definition~\ref{definition sph pair} with the roles of the functors $\delta$ and $\gamma$ exchanged.
\end{corollary}

\section{Intersection cohomology}
\label{section intersection cohomology}

Assume given a spherical pair $\cP$ from Theorem~\ref{theorem.sph_pair}. This section studies the perverse sheaves of vector spaces $\cP^{\K}$ and ${}^{\K}\cP$ obtained by taking complexified Grothendieck groups. The following Theorem~\ref{theorem.Ktheory} gives a necessary and sufficient condition for~${}^{\K}\cP$ to be an intersection cohomology complex of a local system on the punctured disk. Simple sufficient conditions are then given in Proposition~\ref{proposition.very_balanced_odd_codim}. A dual result for $\cP^{\K}$ is established in Theorem~\ref{theorem.Ktheory_dual}.

\begin{theorem}\label{theorem.Ktheory} In the setting of Theorem~\ref{theorem.sph_pair}, consider the perverse sheaf~\,${}^{\K}\cP$ provided by Proposition~\ref{proposition.K_theory}, with description
\[
\begin{tikzpicture}
	\node (zero) at (0,0) {$\K(\cC)$};
	\node (plus) at (3,0) {$\K(X \gitQuotStab{+} G),$};
	\node (minus) at (-3,0) {$\phantom{.}\K(X \gitQuotStab{-} G)$};
	\draw[->,transform canvas={yshift=+\arrowsplit}] (minus) to   (zero);
	\draw[<-,transform canvas={yshift=-\arrowsplit}] (minus) to  node[below] {$\scriptstyle \operatorname{res}_- $} (zero);
	\draw[->,transform canvas={yshift=+\arrowsplit}] (plus) to  (zero);
	\draw[<-,transform canvas={yshift=-\arrowsplit}] (plus) to  node[below] {$\scriptstyle \operatorname{res}_+ $} (zero);
\end{tikzpicture}
\]
where unmarked arrows are given by right adjoints. Then:

\begin{enumerate}
\item\label{theorem.Ktheory 2} the monodromy $m$ around $0$ on $\K(X \gitQuotStab{+} G)$ is given by the action of the twist~$\twistFun_\sphFun$ of a spherical functor 
\begin{equation}\label{equation sph functor}
\sphFun =  \operatorname{res}_+ \circ \, \iota_{-} \colon \cD  \longrightarrow \D(X \gitQuotStab{+} G)
\end{equation}
where $\cD$ denotes the category $\D(Z/L)_-^w;$
\item\label{theorem.Ktheory 3} we have
\begin{equation}\label{equation crucial inequality}
 \operatorname{rk} (m - \Idmatrix) \leq \operatorname{dim} \K( \cD );
\end{equation}
\item\label{theorem.Ktheory 4} there is an isomorphism
\[
{}^{\K}\cP \cong \operatorname{IC}( M )
\]
for $M$ a local system on $\Delta-0$ if and only if the inequality \eqref{equation crucial inequality} is saturated.
\end{enumerate}

\begin{proof}
(1) This is an application of Proposition~\ref{proposition.spherical_functor_monodromy}(\ref{proposition.spherical_functor_monodromy 1}).
\\
(2) Combining (\ref{theorem.Ktheory 2}) with the definition of the twist in Proposition~\ref{proposition.spherical_pair_to_functor}(\ref{proposition.spherical_pair_to_functor 1}), it follows that
\[
m = \K(\twistFun_\sphFun) = \Idmatrix - \K(\sphFun) \K(\sphFun^*)
\]
and hence $m - \Idmatrix$ factors through the $\K(\cD)$, implying~(\ref{equation crucial inequality}).
\\
(3) The only possible candidate for $M$ is the local system on $\Delta-0$ obtained by restricting ${}^{\K}\cP$. Taking this $M$, the claimed equation is true after restriction to $\Delta-0$ by construction. By~Proposition~\ref{proposition.simple_perverse_sheaves}, the simple perverse sheaves on~$\Delta$, possibly singular at~$0$, are of the form $\underline{\field}_0[-1]$ or $\operatorname{IC}(L)$ for an irreducible local system $L$ on $\Delta-0$. Hence the composition factors of~${}^{\K} \cP$ include all the composition factors of~$\operatorname{IC}(M)$, and possibly some number of factors of the form $\underline{\field}_0[-1]$: I argue that this number is zero if and only if the inequality~\eqref{equation crucial inequality} is saturated, and thereby deduce the result.

For $P \in \operatorname{Per} (\mathbb{C},0)$ write $\operatorname{dim}_0 P$   for the dimension of its space of vanishing cycles at $0$, namely the space $D_0$ of Proposition~\ref{proposition.GGM_perverse_description}. This dimension is additive on extensions of perverse sheaves, and from Proposition~\ref{proposition.simple_perverse_sheaves} we see that $\operatorname{dim}_0 \underline{\field}_0[-1] = 1$. It follows that the number of  composition factors of the form $\underline{\field}_0[-1]$ in ${}^{\K} \cP$ is
\[
\operatorname{dim}_0 {}^{\K} \cP - \operatorname{dim}_0 \operatorname{IC}(M) \geq 0.
\]
From Proposition~\ref{proposition.KS_perverse_description} it may be deduced that \[\operatorname{dim}_0 {}^{\K}\cP = \dim \K(\cC) - \dim \K(X \gitQuotStab{\pm} G),\]
but this is just $\dim \K(\cD)$ because the dimension of Grothendieck groups is additive for semi-orthogonal decompositions, in particular the decomposition of Proposition~\ref{proposition sod}. From Proposition~\ref{proposition.simple_perverse_sheaves} we have that
\begin{equation*}
\operatorname{dim}_0 \operatorname{IC}(M) =\operatorname{codim} \operatorname{ker} (m - \Idmatrix) =\operatorname{rk} (m - \Idmatrix)
\end{equation*}
and, combining, the result follows.
\end{proof}
\end{theorem}

The inequality \eqref{equation crucial inequality} may not be saturated, for indeed it often happens that $m = \Idmatrix$ so that $\operatorname{rk}(m-\Idmatrix)=0$ as Example~\ref{example standard flop} shows. On the other hand, we can give easily checkable sufficient conditions for \eqref{equation crucial inequality} to be saturated. The assumptions of Theorem~\ref{theorem.sph_pair} implied that $\eta_+=\eta_-$ so that the $\lambda_\pm$-weights of $\det \cN^\vee_{S_\pm} X$ on $Z$ were equal: under the following slightly stronger condition the spherical pair obtained in Theorem~\ref{theorem.sph_pair} will be easy to control at the level of Grothendieck groups, so that these sufficient conditions for the saturation of~(\ref{equation crucial inequality}) will become apparent.

\begin{lemma}\label{proposition.very_balanced} For $G$ abelian, the following conditions are equivalent.
\begin{enumerate}
\item The line bundle $\det \cN_Z  X$ on $Z$ is $G$-equivariantly trivial.
\item The line bundle
\begin{equation}\label{equation.line_bundle}
\det \cN^{\phantom{\vee}}_{S_+} \!X \,\otimes\, \det \cN^{\phantom{\vee}}_{S_-} \!X
\end{equation}
is $G$-equivariantly trivial after restriction to $Z$.
\end{enumerate}
These conditions hold if $X$ and $Z$ are $G$-equivariantly Calabi--Yau.

\begin{proof}The equivalence claim follows from the identification of the strata $S_\pm$ with the $Y_\pm$ and the splitting, as in~(\ref{splitting of normal bundle Z}), of~$\Omega_{X}$ on~$Z$ into
\[ \cN^{\vee}_{Y_-} \!X|_Z \,\oplus\, \Omega_Z \,\oplus\, \cN^{\vee}_{Y_+} \!X|_Z, \]
combined with the definition of $\cN_Z  X$.
The last claim follows from the adjunction formula
\[ \det \cN_Z X = \omega_Z \otimes \omega_X^\vee |_Z.\qedhere \]
\end{proof}
\end{lemma}

\begin{proposition}\label{proposition.very_balanced_odd_codim} In the setting of Theorem~\ref{theorem.sph_pair}, assume furthermore that: 
\renewcommand{\theenumi}{\roman{enumi}}
\begin{enumerate}
\item the equivalent conditions of Lemma~\ref{proposition.very_balanced} hold; and
\item the codimension of $Z$ in $X$ is odd.
\end{enumerate}
Then the inequality \emph{(\ref{equation crucial inequality})} is saturated, and \[{}^{\K}\cP  \cong \operatorname{IC}( M )\] where $M$ is the local system on $\Delta-0$ obtained by restricting ${}^{\K}\cP.$

\begin{proof}
I first calculate the action of the cotwist~$\cotwistFun_\sphFun$ of the spherical functor~$\sphFun$ from~(\ref{equation sph functor}) on~ $\K (\cD)$ under the assumptions. From Proposition~\ref{proposition.spherical_pair_to_functor}(\ref{proposition.spherical_pair_to_functor 2}) this is given by the composition~(\ref{sph pair functor 1}) then~(\ref{sph pair functor 2}) from Definition~\ref{definition sph pair}. In our setting, by Theorem~\ref{theorem.sph_pair}(\ref{theorem.sph_pair 3}), this acts by tensoring by the line bundle \eqref{equation.line_bundle} restricted to~$Z$, which is $G$-equivariantly trivial by assumption, and a cohomological shift of the negative~of
\[ \operatorname{codim}_X S_+ \,+\, \operatorname{codim}_X S_-. \]
For $G$ abelian, the strata $S_\pm$ are identified with the $Y_\pm$, which intersect transversally in $Z$ as explained in the proof of Lemma~\ref{lemma compute compositions}, so this is just $\operatorname{codim}_X Z$. If this is odd, then the action of $\cotwistFun_\sphFun$ on the Grothendieck group is therefore simply $-\Idmatrix$. It then follows from the definition of $\cotwistFun_\sphFun$ in Proposition~\ref{proposition.spherical_pair_to_functor}(\ref{proposition.spherical_pair_to_functor 2}) that
\[\K(\cotwistFun_\sphFun) = \Idmatrix - \K(\sphFun^*) \K(\sphFun) = - \Idmatrix \qquad \text{and} \qquad \K(\sphFun^*) \K(\sphFun) = 2 \cdot \Idmatrix .\]
The map $2 \cdot \Idmatrix$ is clearly bijective, and thence $\K(\sphFun)$ is injective, and $\K(\sphFun^*)$ surjective onto $\K (\cD)$. Using Theorem~\ref{theorem.Ktheory}(\ref{theorem.Ktheory 2}) and the definition of~$\twistFun_\sphFun$ from Proposition~\ref{proposition.spherical_pair_to_functor}(\ref{proposition.spherical_pair_to_functor 1}) we then have that
\begin{align*}\operatorname{rk} (m - \Idmatrix) & = \operatorname{rk} \!\big(\K(\twistFun_\sphFun) - \Idmatrix\,\big) \\
& = \operatorname{rk} \!\big(\!-\K(\sphFun) \K( \sphFun^*)\big) \\
& = \operatorname{dim} \K (\cD),\end{align*} and hence the inequality \eqref{equation crucial inequality} is saturated, and the claim follows by the proof of Theorem~\ref{theorem.Ktheory}(\ref{theorem.Ktheory 4}). 
\end{proof}
\end{proposition}

Finally, we record the following dual result to Theorem~\ref{theorem.Ktheory}.

\begin{theorem}\label{theorem.Ktheory_dual} In the setting of Theorem~\ref{theorem.sph_pair}, consider the perverse sheaf~$\cP^{\K}$ provided by Proposition~\ref{proposition.K_theory}, with description
\[
\begin{tikzpicture}
	\node (zero) at (0,0) {$\K(\cC)$};
	\node (plus) at (3.5,0) {$\Kbig{\D(Z/L)_+^{-w-\eta}},$};
	\node (minus) at (-3.25,0) {$\phantom{.}\Kbig{\D(Z/L)_{-\vphantom{+}}^{w\vphantom{\eta}}}$};
	\draw[->,transform canvas={yshift=+\arrowsplit}] (minus) to  node[above] {$\scriptstyle \iota_{-\vphantom{+}} $} (zero);
	\draw[<-,transform canvas={yshift=-\arrowsplit}] (minus) to   (zero);
	\draw[->,transform canvas={yshift=+\arrowsplit}] (plus) to  node[above] {$\scriptstyle \iota_+ $} (zero);
	\draw[<-,transform canvas={yshift=-\arrowsplit}] (plus) to  (zero);
\end{tikzpicture}
\]
where unmarked arrows are given by left adjoints. Then:

\begin{enumerate}
\item the monodromy $m'$ around~$0$ on~$\Kbig{\D(Z/L)_-^w}$ is given by the action of the cotwist~$\cotwistFun_\sphFun$ of the spherical functor \eqref{equation sph functor};
\item we have
\begin{equation}\label{equation crucial inequality dual}
 \operatorname{rk} (m' - \Idmatrix) \leq \operatorname{dim} \K(X \gitQuotStab{+} G);
\end{equation}
\item there is an isomorphism
\[
\cP^{\K} \cong \operatorname{IC}( M' )
\]
for $M'$ a local system on $\Delta-0$ if and only if the inequality \emph{(\ref{equation crucial inequality dual})} is saturated.
\end{enumerate}
\begin{proof}The argument is dual to that of Theorem~\ref{theorem.Ktheory}, using Proposition~\ref{proposition.spherical_functor_monodromy}(\ref{proposition.spherical_functor_monodromy 2}) in place of Proposition~\ref{proposition.spherical_functor_monodromy}(\ref{proposition.spherical_functor_monodromy 1}), and so on.
\end{proof}
\end{theorem}

\newpage
\section{Examples}
\label{section examples}

In this section, I give some examples of spherical pairs for flops in higher dimensions.
\begin{example}\label{example orbifold} Let $X=V \oplus \det V^\vee$ for $V$ a vector space of dimension $d=2n$, with $\mathbb{C}^*$-action induced by dilation of $V$. Then $Z=0$ with
\[
X^{\mathrm{ss}}_+ = X - (0 \oplus \det V^\vee) \qquad \text{and} \qquad
X^{\mathrm{ss}}_- = X - (V \oplus 0),
\]
so that the quotient $X \gitQuotStab{+} \mathbb{C}^*$ is the total space of the bundle \[\cO(-2n) \to \mathbb{P}V,\] and $X \gitQuotStab{-} \mathbb{C}^*$ is its flop, the orbifold $V/C_{2n}$ of the cyclic group $C_{2n}$ acting on $V$ by scalars.

The space $X$~is smooth of dimension~\mbox{$2n+1$} and equivariantly Calabi--Yau, so the assumptions of Theorem~\ref{theorem.sph_pair} and Proposition~\ref{proposition.very_balanced_odd_codim} are easily seen to be satisfied. Hence we have a spherical pair $\cP$, and ${}^{\K}\cP$ is an intersection cohomology complex of a local system on $\Delta-0$.  \end{example}

\begin{example}\label{example standard flop} Let $X=V \oplus V^\vee$ for $V$ a vector space of dimension $d$, with $\mathbb{C}^*$-action induced by dilation of $V$. Then $Z=0$ with
\[
X^{\mathrm{ss}}_+ = X - (0 \oplus V^\vee) \qquad \text{and} \qquad
X^{\mathrm{ss}}_- = X - (V \oplus 0),
\]
so that the quotient $X \gitQuotStab{+} \mathbb{C}^*$ is the total space of the bundle $V^\vee \otimes \cO(-1) \to \mathbb{P}V$, and $X \gitQuotStab{-} \mathbb{C}^*$ is its flop, the total space of $V \otimes \cO(-1) \to \mathbb{P}V^\vee$.

Once again, the space $X$ is smooth and equivariantly Calabi--Yau, so the assumptions of Theorem~\ref{theorem.sph_pair} are easily seen to be satisfied. Hence we have a spherical pair $\cP$. However, we cannot apply Proposition~\ref{proposition.very_balanced_odd_codim} for dimension reasons and indeed the conclusion, that ${}^{\K}\cP$ is an intersection cohomology complex of a local system on $\Delta-0$, does not hold.

To see this, let $w=0$ so that $\iota_-^w(\cO_0) = \cO_{V \oplus 0}$, and observe (for instance by comparing terms in the respective Koszul resolutions) that this has the same Grothendieck group class, up to sign, as $ \cO_{0 \oplus V^\vee} \otimes \det V$. The latter object is clearly in the kernel of $\operatorname{res}_+$, and because $\cO_0$ is a generator we may deduce that $\K(\sphFun)=0$ and 
\begin{align*}\operatorname{rk} (m - \Idmatrix) & = \operatorname{rk} \!\big(\K(\twistFun_\sphFun) - \Idmatrix\,\big) \\
& = \operatorname{rk} \!\big(\!-\K(\sphFun) \K( \sphFun^*)\big) \\
& = 0.\end{align*}
Here $\cD=D(Z/L)^w_-$ is generated by a single object $\cO_0$, so \[\operatorname{dim} \K( \cD )=1,\] the inequality~\eqref{equation crucial inequality} is not saturated, and thence ${}^{\K}\cP$ is not the intersection cohomology complex of a local system on $\Delta-0$ by Theorem~\ref{theorem.Ktheory}(\ref{theorem.Ktheory 4}).\end{example}



\newpage
\begin{thebibliography}{99}
\bibitem{Addington}
N.~Addington, \emph{New derived symmetries of some hyperk\"ahler varieties}, Alg.\ Geom.\ \textbf{3} (2) (2016) 223--260, \hrefsf{http://arxiv.org/abs/1112.0487}{arXiv:1112.0487}.

\bibitem{Anno} 
R.~Anno, \emph{Spherical functors}, \hrefsf{http://arxiv.org/abs/0711.4409}{arXiv:0711.4409}.

\bibitem{AL2}
R.~Anno and T.~Logvinenko, \emph{Spherical DG functors}, \hrefsf{http://arxiv.org/abs/1309.5035}{arXiv:1309.5035}, to appear J.\ Eur.\ Math.\ Soc.

\bibitem{Arinkin-Caldararu}
D.~Arinkin and A.~C\u{a}ld\u{a}\-raru, \emph{When is the self-intersection of a subvariety a fibration?}, \hrefsf{http://arxiv.org/abs/1007.1671}{arXiv:1007.1671}.

\bibitem{BFK}
M.~Ballard, D.~Favero, and L.~Katzarkov, \emph{Variation of geometric invariant theory quotients and derived categories}, \hrefsf{http://arxiv.org/abs/1203.6643}{arXiv:1203.6643}, to appear J.\ Reine Angew.\ Math.


\bibitem{Beilinson1984}
A.~Beilinson, \emph{How to glue perverse sheaves}, in K-theory, arithmetic and geometry (Moscow, 1984), Lecture Notes in Math. \textbf{1289}, Springer-Verlag, 1987, 42--51.

\bibitem{BBD1982}
A.~Beilinson, J.~Bernstein, and P.~Deligne, Faisceaux pervers, Ast\'erisque (Soci\'et\'e Math\'ematique de France, Paris) \textbf{100}, 1982\hrefsfopt{, available at }{https://webusers.imj-prg.fr/~olivier.dudas/src/faisceaux-pervers.pdf}{webusers.imj-prg.fr}.

\bibitem{BL}
J.~Bernstein and V.~Lunts, Equivariant Sheaves and Functors, Springer Lecture Notes in Math.\ \textbf{1578} (1994), Springer.

\bibitem{BBirula}
A.~Bia\l{}ynicki-Birula, \emph{Some theorems on actions of algebraic groups}, Ann.\ Math.\ \textbf{98} (3) (1973), 480--497.

\bibitem{BB}
A.~Bodzenta and A.~Bondal, \emph{Flops and spherical functors}, \hrefsf{http://arxiv.org/abs/1511.00665}{arXiv:1511.00665}.

\bibitem{BKS}
A.~Bondal, M.~Kapranov, and V.~Schechtman, \emph{Perverse schobers and birational geometry}, in preparation.

\bibitem{DH}
I.~V.~Dolgachev and Y.~Hu, \emph{Variation of geometric invariant theory quotients}, Publications Math\'ematiques de l'IH\'ES \textbf{87} (1998), 5--51.


\bibitem{DK}
T.~Dyckerhoff and M.~Kapranov, \emph{Triangulated surfaces in triangulated categories}, \hrefsf{http://arxiv.org/abs/1306.2545}{arXiv:1306.2545}.

\bibitem{DKSS}
T.~Dyckerhoff, M.~Kapranov, V.~Schechtman, and Y.~Soibelman, \emph{Perverse schobers on surfaces and Fukaya categories with coefficients}, in preparation.

\bibitem{GGM}
A.~Galligo, M.~Granger, and Ph.~Maisonobe, \emph{D-modules et faisceaux pervers dont le support singulier est un croisement normal}, Ann.\ Inst.\ Fourier (Grenoble), \textbf{35}:1--48, 1985.


\bibitem{HL}
D.~Halpern-Leistner, \emph{The derived category of a GIT quotient}, J.\ Amer.\ Math.\ Soc., \textbf{28} (3) (2015) 871--912, \hrefsf{http://arxiv.org/abs/1203.0276}{arXiv:1203.0276}.

\bibitem{HLShipman}
D.~Halpern-Leistner and I.~Shipman,
\emph{Autoequivalences of derived categories via geometric invariant theory}, Adv.\ Math.\ \textbf{303} (2016) 1264--1299, \hrefsf{http://arxiv.org/abs/1303.5531}{arXiv:1303.5531}.

\bibitem{HKL}
A.~Harder, L.~Katzarkov, and Y.~Liu, \emph{Perverse sheaves of categories and non-rationality}, in Geometry over nonclosed fields (F.~Bogomolov, B.~Hassett, and Y.~Tschinkel, editors), Springer, 2017, 53--96\hrefsfopt{, available at }{http://www.cims.nyu.edu/~tschinke/books/simons15/submitted/harder-katzarkov.pdf}{www.cims.nyu.edu}.

\bibitem{HHP}
M.~Herbst, K.~Hori, and D.~Page, \emph{Phases of $\mathcal{N}=2$ theories in $1+1$ dimensions with boundary}, \hrefsf{http://arxiv.org/abs/0803.2045}{arXiv:0803.2045}.

\bibitem{HuybrechtsFM}
D.~Huybrechts.
\newblock {Fourier--Mukai transforms in algebraic geometry}, Oxford Mathematical Monographs. The Clarendon Press, Oxford University Press, Oxford, 2006. viii+307 pp.

\bibitem{Kirwan}
F.~C.~Kirwan.
\newblock {Cohomology of quotients in symplectic and algebraic geometry}, Math.\ Notes \textbf{31}, Princeton University Press, 1984.

\bibitem{KS1}
M.~Kapranov and V.~Schechtman, \emph{Perverse sheaves over real hyperplane arrangements}, Ann.\ Math.\ \textbf{183} (2) (2016) 619--679, \hrefsf{http://arxiv.org/abs/1403.5800}{arXiv:1403.5800}.

\bibitem{KS2}
\bysame, \emph{Perverse Schobers}, \hrefsf{http://arxiv.org/abs/1411.2772}{arXiv:1411.2772}.

\bibitem{KS3}
\bysame, \emph{Perverse sheaves and graphs on surfaces}, \hrefsf{http://arxiv.org/abs/1601.01789}{arXiv:1601.01789}.

\bibitem{Kontsevich-Symplectic}
M.~Kontsevich, \emph{Symplectic geometry of homological algebra}, available at \hrefsf{http://www.ihes.fr/~maxim/TEXTS/Symplectic_AT2009.pdf}{www.ihes.fr}.

\bibitem{Kontsevich-Talk10}
\bysame, \emph{Singular Lagrangian branes}, talk at Homological Mirror Symmetry and Related Topics, University of Miami, 2010\hrefsfopt{, abstract available at }{https://math.berkeley.edu/~auroux/frg/miami10-abstracts.html}{math.berkeley.edu}.

\bibitem{Kontsevich-Talk15}
\bysame, \emph{Mirror symmetry: new definitions}, talk at AMS Summer Institute in Algebraic Geometry, 2015\hrefsfopt{, }{https://sites.google.com/site/2015summerinstitute/home/schedule/week-2-abstracts}{abstract}\hrefsfopt{ and }{https://docs.google.com/viewer?a=v&pid=sites&srcid=ZGVmYXVsdGRvbWFpbnwyMDE1c3VtbWVyaW5zdGl0dXRlfGd4OjE2ZWM2ZWU2ZGQ5MTYyMDg}{slides}\hrefsfopt{ available at }{https://sites.google.com/site/2015summerinstitute/home/schedule/week-2-abstracts}{sites.google.com/site/2015summerinstitute}. 

\bibitem{PS}
J.~Pascaleff and N.~Sibilla, \emph{Topological Fukaya category and mirror symmetry for punctured surfaces}, \hrefsf{https://arxiv.org/abs/1604.06448}{arXiv:1604.06448}.

\bibitem{NadlerArb}
D.~Nadler, \emph{Arboreal singularities}, \hrefsf{https://arxiv.org/abs/1309.4122}{arXiv:1309.4122}, to appear in Geom. Topol.

\bibitem{Nadler}
\bysame, \emph{Mirror symmetry for the Landau-Ginzburg A-model $M=\mathbb{C}^n$, $W=z_1 \dots z_n$,} \hrefsf{https://arxiv.org/abs/1601.02977}{arXiv:1601.02977}.

\bibitem{Rietsch}
K.~Rietsch, \emph{An Introduction to Perverse Sheaves}, Fields Institute Communications \textbf{40} (2004), 391--429, \hrefsf{http://arxiv.org/abs/math/0307349}{arXiv:math/0307349}.

\bibitem{Seg}
E.~Segal, \emph{Equivalences between GIT quotients of Landau-Ginzburg B-models}, Comm.\ Math.\ Phys.\ \textbf{304} (2011), 411--432, \hrefsf{http://arxiv.org/abs/0910.5534}{arXiv:0910.5534}.

\bibitem{Seg2}
\bysame, \emph{All autoequivalences are spherical twists}, \hrefsf{http://arxiv.org/abs/1603.06717}{arXiv:1603.06717}, to appear Int.\ Math.\ Res.\ Not.

\bibitem{Seidel}
P.~Seidel, Fukaya Categories and Picard--Lefschetz Theory, European Mathematical Society (EMS), Zurich Lectures in Advanced Mathematics, Z\"urich, 2008.


\bibitem{Soibelman}
Y.~Soibelman, \emph{Perverse sheaves of categories and wall-crossing formulas}, talk at Simons Center, November 2014\hrefsfopt{, slides available at }{http://media.scgp.stonybrook.edu/presentations/20141118_Soibelman.pdf}{media.scgp.stonybrook.edu}.

\bibitem{Thaddeus}
M.~Thaddeus, \emph{Geometric invariant theory and flips}, J.\ Amer.\ Math.\ Soc., \textbf{9} (3) (1996) 691--723, \hrefsf{http://arxiv.org/abs/alg-geom/9405004}{arXiv:alg-geom/9405004}.
\end{thebibliography}
\end{document}